\documentclass[12pt,a4paper]{iopart}
\usepackage{iopams}
\expandafter\let\csname equation*\endcsname\relax
\expandafter\let\csname endequation*\endcsname\relax
\usepackage{amsmath}
\usepackage[utf8]{inputenc}
\usepackage{xcolor}
\usepackage{amsthm}
\usepackage{graphicx}
\usepackage{subfigure} 
\usepackage{mathtools}
\usepackage{enumitem}
\usepackage{nicefrac} 
\usepackage{cleveref}
\usepackage{siunitx}
\usepackage{float}
\usepackage{capt-of}
\usepackage[language=english,backend=biber,style=alphabetic,hyperref, maxbibnames=999, maxcitenames=5]{biblatex}

\renewcommand{\epsilon}{\varepsilon}
\newcommand{\CC}{\ensuremath{\mathbb{C}}}
\newcommand{\dd}{\ensuremath{\;\mathrm{d}}}
\newcommand{\loc}{\ensuremath{\mathrm{loc}}}
\newcommand{\R}{\ensuremath{\mathbb{R}}}
\newcommand{\tx}{\ensuremath{\underline{x}}}

\newcommand{\J}{\ensuremath{\mathcal{J}_{\R^{d-1}}}}
\newcommand{\JRand}{\ensuremath{\mathcal{J}}}

\renewcommand*{\Re}{\ensuremath{\mathrm{Re}~}}
\renewcommand*{\Im}{\ensuremath{\mathrm{Im}~}}
\renewcommand*{\tilde}{\widetilde}

\newcommand\ii{\ensuremath{\mathrm i}}
\newcommand\N{\ensuremath{\mathbb N}}
\newcommand\Q{\ensuremath{\mathbb Q}}
\newcommand\Z{\ensuremath{\mathbb Z}}
\newcommand\K{\ensuremath{\mathbb K}}
\renewcommand*{\hat}{\widehat}

\newtheorem{theorem}{Theorem}
\newtheorem{lemma}[theorem]{Lemma}
\newtheorem{corollary}[theorem]{Corollary}
\newtheorem{definition}[theorem]{Definition}
\newtheorem{assumption}{Assumption}
\newtheorem{problem}{Problem}
\newtheorem{remark}[theorem]{Remark}
\newtheorem{proposition}[theorem]{Proposition}

\begin{document}

\title[Reconstruction of Perturbation in Inhomogeneous Periodic Layer]{Reconstruction of a  Local Perturbation in Inhomogeneous Periodic Layers from Partial Near Field Measurements}

\author{Alexander Konschin$^1$ and Armin Lechleiter}

\address{$^1$ RTG 2224 ``Parameter Identification - Analysis, Algorithms, Applications'' and Center for Industrial Mathematics, University of Bremen, Germany}
\ead{alexk@uni-bremen.de}
\vspace{10pt}
\begin{indented}
\item[]\today
\end{indented}

\begin{abstract}
    We consider the inverse scattering problem to reconstruct a local perturbation of a given inhomogeneous periodic layer in $\R^d$, $d=2,3$, using near field measurements of the scattered wave on an open set of the boundary above the medium, or, the measurements of the full wave in some area. The appearance of the perturbation prevents the reduction of the problem to one periodic cell, such that classical methods are not applicable and the problem becomes more challenging. 		
    We first show the equivalence of the direct scattering problem, modeled by the Helmholtz equation formulated on an unbounded domain, to a family of quasi-periodic problems on a bounded domain, for which we can apply some classical results to provide unique existence of the solution to the scattering problem.
    The reformulation of the problem is also the key idea for the numerical algorithm to approximate the solution, which we will describe in more detail. Moreover, we characterize the smoothness of the Bloch-Floquet transformed solution of the perturbed problem w.r.t.\ the quasi-periodicity to improve the convergence rate of the numerical approximation.
    Afterward, we define two measurement operators, which map the perturbation to some measurement data, and show uniqueness results for the inverse problems, and the ill-posedness of these. 
    Finally, we provide numerical examples for the direct problem solver as well as examples of the reconstruction in 2D and 3D.
\end{abstract}

	\section{Introduction}
	{
        
        The growing industrial interest for micro or nano-structured materials and the resulting challenge to construct an automated non-destructing testing method for the structures is one of the fundamental motivations to study perturbed periodic scattering problems.
        The direct and inverse scattering problems from unbounded periodic structures is a well-established topic in mathematics, especially if one considers quasi-periodic incident fields. This assumption allows to reduce the problem on the infinite periodic domain into one periodic cell, such that standard techniques for the existence theory and the standard numerical methods for bounded domains can be applied (see, e.g., \cite{Bonne1994}, \cite{Dobson1992}, \cite{Abboud1992}, \cite{BaoDobsomCox95}, \cite{Bao1994}, \cite{Bao1995}, \cite{Kirsc1993}, \cite{Kirsc1995a}). If the periodicity is perturbed, or, one uses non-periodic incident fields, such as Gaussian beams, the reduction is typically impossible and one has to treat the problem as a scattering problem for an unbounded rough layer 
        (see, e.g., \cite{Hadda2011}, \cite{HuLiuQuZhang2015}, \cite{ArensMeier2000}). The disadvantage is that for the existence theory one has to assume more regularity for the parameter, which we can avoid by considering the periodicity of the unperturbed parameter and applying the Bloch-Floquet transform to the variational problem to get an alternative problem. There are, however, some approaches for problems on locally perturbed periodic waveguides based on the Bloch-Floquet transform, see \cite{Joly2006}, \cite{Fliss2015}, \cite{Ehrhardt2009}.
        
        In this paper, we study the scattering problem formulated in the upper half space $\R_+^d := \{ x \in \R^d : x_d > 0\}$, $d=2,3$,
        \begin{align*}
            \Delta u + k^2 n^2 u
            &= -f
            &&\text{ in }
            \R_+^d
            ,
            \\
            u
            &= 0
            &&\text{ on }
            \{x_d = 0\}
            ,
        \end{align*}
        for a locally perturbed inhomogeneous layer, which is described by the refractive index $n^2 \in L^\infty(\R_+^d)$. Applying the Bloch-Floquet transform to decompose the (non-periodic) incident field into its quasi-periodic components, we can reformulate the scattering problem as a family of quasi-periodic scattering problems on a bounded domain.
        We show equivalence of the two problems and consider the latter to prove existence of the solution to the scattering problem by applying Fredholm theory for the reduced problem. 
        Moreover, we stay in the framework of the equivalent formulation to introduce a numerical method to approximate the solution to the original problem, which is based on \cite{LechleiterZhang2017} and \cite{Zhang2018}, where the algorithm for the sound-soft scattering layer is developed. Considering the regularity of the transformed solution w.r.t.\ the quasi-periodicity, we are able to improve the convergence rate of the inverse Bloch transform, approximated by the trapezoidal rule, and by choosing an adequate variable transform. For the implementation of the direct problem solver, we use the Finite-Element-Method library \emph{deal.II} (\cite{dealII85}). 
        The drawback of this method is that one needs to be able to compute analytically, or numerically, the transformed function of the incident wave. At least for incident point sources and Herglotz wave functions, which are models for Gaussian beams, some semi-analytic expressions are available in \cite{Lechl2015e}.
        
        In the second part, we consider the inverse scattering problem to reconstruct the local perturbation by analyzing the measurement operator $\Lambda : \mathcal{D}(\Lambda) \to \mathcal{L}(L^2(\Omega_0^{R_0}), L^2(\Omega^R_0))$, where $\Omega^R_0$ will be later defined  as one periodic cell  for $R >0$. The operator maps the perturbation $q$ to the solution operator dependent on $q$, which maps right hand sides in $L^2$ supported in $\Omega_0^{R_0}$ to scattered waves restricted to $\Omega^R_0$. 
        Furthermore, we consider the second measurement operator $\mathcal{S}: \mathcal{D}(\mathcal{S}) \to \mathcal{L}(L^2(\Omega_0^{R_0}), L^2(\partial \Omega^R_0 \cap \{x_d = R\}))$ mapping the perturbation to the operator, which maps $L^2(\Omega_0^{R_0})$ right hand sides to the upper trace of the scattered field, also restricted to one periodic cell. 
        We show injectivity of $\mathcal{S}$ in the case that $d=3$, as long as the parameters are twice differentiable, and the whole trace on $\Gamma^R$ is given as data, considering the \emph{complex geometrical optics} (see, e.g., \cite{SylvesterUhlmann1988}). In addition, we show the injectivity of $\Lambda$ (without these restrictions). Moreover, we compute the Fr\'echet derivative of these operators and show that the Fr\'echet derivative is a compact operator and the so-called \emph{tangential cone condition} is satisfied by these operators, such that both inverse problems are locally ill-posed as well as the inverse problem for their linearizations.
        To show some numerical examples, we use the inexact Newton method {CG-REGINN} (\cite{Rieder2005}), to reconstruct the perturbation from artificially generated noisy data.
        
        The Bloch-Floquet transform is a well-known approach in electrical engineering, which is called the \emph{array scanning method}, see, e.g., \cite{MunkBurrell1979}, \cite{Valerio2008}. 
        Nevertheless, the consideration of applying the transform to scattering problems was given just recently by constructing a numerical scheme and analyzing error bounds for the acoustic and electromagnetic scattering problem in the case of sound-soft boundary conditions (see \cite{LechleiterZhang2017}, \cite{Zhang2018}, \cite{LechleiterZhang2017b}). Moreover, in \cite{Hadda2016} the acoustic scattering problem for an inhomogeneous layer was studied by applying the Bloch-Floquet transform and considering integral equations. The setting of the direct problem is close to the one in this paper, but with the somewhat easier assumption of a wave number with a positive imaginary part.

        The remainder of this paper is structured as follows. In  \Cref{Sec21} we consider the direct problem, for which we present the setting of the scattering problem corresponding to the locally perturbed periodic layer. We use the Bloch-Floquet transform to show unique existence of the solution for the unperturbed case in  \Cref{Sec22}  and consider the perturbed layer problem in  \Cref{Sec23}.
        In \Cref{Sec3} we analyze the inverse problem by defining a suitable parameter space, defining the parameter-to-state map, calculating the Fr\'echet derivative and show the ill-posedness as well as the uniqueness for the inverse problems. 
        In the last two sections, we introduce the numerical method for the direct and inverse problem in \Cref{Sec4} and show some numerical examples in \Cref{Sec5}.
	}

	\section{Direct Scattering Problem}\label{Sec2}
    {
   		In this section we formulate the scattering problem for a perturbed periodic layer and prove unique existence of the scattered field. For that, we use the Bloch-Floquet transform to reduce the problem to a family of quasi-periodic problems on a bounded domain. 
	}
	\subsection{Formulation of the problem}\label{Sec21}
    {
        Suppose $n_{p}^2 \in L^\infty(\R_+^d)$, $d = 2,3$, is a $L$-periodic refractive index in $\tx := (x_1, \ldots, x_{d-1})$, which satisfies $n_p^2 = 1$ for $x_d > R_0 > 0$ and characterizes the unperturbed scattering layer. 
        To simplify the notation, we assume that $L$ equals to the scaled identity matrix $2\pi I_{d-1} \in \R^{ (d-1)^2}$ and that the local perturbation $q \in L^\infty(\R_+^{d})$ has the support in $\Omega_0^R := (-\pi, \pi)^{d-1} \times (0, R)$ for $R > R_0$, such that we consider the perturbed refractive index $n^2 := n^2_p + q$.
        Define for $R \geq 0$ the sets
        \begin{align*}
            \Omega^R &:= \R^{d-1} \times (0, R)
            ,
            &\Gamma^R &:= \R^{d-1} \times \{R\}
            ,
            \\
            \Gamma_0^R &:= (-\pi, \pi)^{d-1} \times \{R\}
            \text{ and }
            &I_{} &:= (\nicefrac{-1}{2}, \nicefrac{1}{2})^{d-1}
            .
        \end{align*}
        
        The scattering problem is to find the scattered field $u \in H_{0,\loc}^1(\R_+^d) \cap H^1(\Omega^R)$ for every $R > R_0$, such that
        \[
            \Delta u + k^2 n^2 u 
            = -f
            \text{ in }
            \R_+^d
            ,
            \quad
            u = 0
            \text{ on }
            \Gamma^0
            .
        \]
        Moreover, the scattering field is assumed to satisfy the so-called angular spectrum representation
        \begin{equation}\label{eq_radiationCond}
            u(x)
            := \frac{1}{(2\pi)^{\nicefrac{(d-1)}{2}}} \int_{\R^{d-1}} e^{\ii \tx \cdot \xi + \ii \sqrt{k^2 - |\xi|^2} (x_d - R)} \hat{u}(\xi, R) \dd \xi
            \quad \text{for }
            x_d > R
            ,
        \end{equation}
        where $\hat{u}$ is the Fourier transform of $u\big|_{\Gamma^{R}}$ and the square root is extend by a branch cut at the negative imaginary axis. As a consequence, we can define the exterior Dirichlet-to-Neumann map $T$ as
        \begin{equation}\label{eq_DtN1}
            \frac{\partial u}{\partial x_d}(\tx,  R)
            = \frac{1}{(2\pi)^{\nicefrac{(d-1)}{2}}} \int_{\R^{d-1}} \ii \sqrt{k^2-|\xi|^2} e^{\ii \tx \cdot \xi} \hat{u}(\xi, R) \dd \xi 
            =: T(u\big|_{\Gamma^R})(\tx)
            ,
        \end{equation}
        which is a bounded linear operator from $H^{\nicefrac{1}{2}}(\Gamma^R)$ to $H^{\nicefrac{-1}{2}}(\Gamma^R)$.
        
        The analysis is easily extendable to the setting of free space scattering problem, assuming that the scattered field satisfies the angular spectrum representation in both directions. From now on, we call the space of $H^1(\Omega^R)$-functions with vanishing trace on $\Gamma^0$ as $\tilde{H}^1(\Omega^R)$ and we consider an arbitrary function $f \in L^2(\Omega^R)$, thus, the variational formulation is to 
        \\
        \begin{problem}\label{prob_Var1}
            Find a function  $u \in \tilde{H}^1(\Omega^R)$, such that
            \begin{equation}\label{eq_Var1}
                a_q(u, v)
                := \int_{\Omega^R} \nabla u \cdot \nabla \overline{v} - k^2 n^2 u \overline{v} \dd x - \int_{\Gamma^R} T(u \big|_{\Gamma^R}) \overline{v} \dd S
                = \int_{\Omega^R} f \overline{v} \dd x
            \end{equation}
            for all $v \in \tilde{H}^1(\Omega^R)$, where $n^2 = n_p^2+q \in L^\infty(\Omega^R)$.
        \end{problem}
        
        Since for real wave numbers $k$ and for a real refractive index some surface waves can exist, we assume a small area of absorption.
        
        \begin{assumption}\label{assumption_absorption}
        	The set $\{\Im n_p^2 > 0\}$ is not empty and contains an open subset. Moreover, it holds $\Im n_p^2 \geq 0$ and $\Im q \geq 0$.
        \end{assumption}
        
        The main result for this section is to prove unique existence of the scattered field.
        \begin{theorem}\label{thm_whole}
            If the \Cref{assumption_absorption} holds, the variational problem \ref{prob_Var1} has a unique solution.
        \end{theorem}
        To prove the theorem, we consider the quasi-periodic problem first.
		\begin{figure}[!ht]
			\center
            \includegraphics[width=0.94\textwidth]{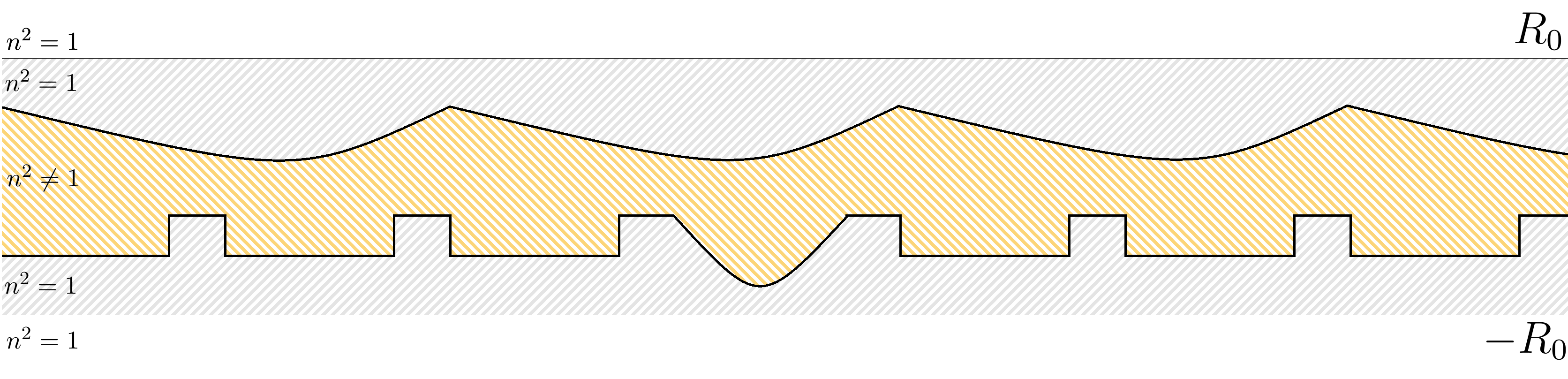}
			\caption{Example for the refractive index $n^2 = n^2_p + q$.}
			\label{image_setting}
		\end{figure}
    }

	\subsection{Quasi-periodic inhomogeneous layer scattering}\label{Sec22}
    {
        In this subsection we will be concerned with the quasi-periodic scattering problem and show the equivalence of the variational problem \ref{prob_Var1} to a family of quasi-periodic problems applying the Bloch-Floquet transform. For that, we treat the case that there is no perturbation at first, that means that $q=0$ and $n^2 = n^2_p$.
        A function is called $\alpha$-quasi-periodic with $\alpha \in \R^{d-1}$ and period $2 \pi$, if
        \[
            u(\tx + 2 \pi j, x_d)
            = e^{-2 \pi \ii \alpha \cdot j} u(\tx, x_d)
            \quad\text{for all }
            j \in \Z^{d-1}
            .
        \]
        For smooth functions $\phi \in C^\infty_0(\overline{\Omega^R})$, the horizontal Bloch-Floquet transform $\J$ is defined by
        \[
            \J \phi (\alpha, \tx, x_d)
            :=  \sum_{j \in \Z^{d-1}} \phi (\tx + 2 \pi j, x_d) e^{2 \pi \ii \alpha \cdot j}
            .
        \]
        Recall the spaces $H^s_\alpha(\Omega_0^R)$ and $H^s_\alpha(\Gamma_0^R)$ of $\alpha$-quasi-periodic Sobolev functions, and set $\tilde{H}^s_\alpha(\Omega_0^R)$ as the subspace of functions $u \in H^s_\alpha(\Omega_0^R)$, such that $u \big|_{\Gamma_0^0} = 0$. The Bloch-Floquet transform extends for $s \in \R$ to an isomorphism between $\tilde{H}^s(\Omega^R)$ and $L^2(I_{}; \tilde{H}^s_\alpha(\Omega_0^R))$ as well as between ${H}^s(\Gamma^R)$ and $L^2(I_{}; {H}^s_\alpha(\Gamma_0^R))$, where the index $\alpha$ indicates that the space depends on $\alpha \in I_{}$ (see \cite{Lechl2016}).
        The inverse of the transform is given by
        \[
            \J^{-1} w
            (\tx + 2 \pi j, x_d)
            =  \int_{I_{}} w(\alpha, \tx, x_d) e^{-2 \pi \ii \alpha \cdot j}
            \dd \alpha
            ,
            \ 
            x \in \Omega_0^R
            ,
            \ 
            j \in \Z^{d-1}
            .
        \]
        
        The scattered field $u_{\alpha} \in \tilde{H}^s_\alpha(\Omega_0^R)$ of the quasi-periodic scattering problem should satisfy the Rayleigh radiation condition
        \begin{equation}\label{eq_alphaRadiationCond}
            u_{\alpha}(\tx, x_d)
            := \sum_{j \in \Z^{d-1}} \widehat{(u_{\alpha}\big|_{\Gamma_0^R})}_j e^{-\ii \alpha_j \cdot \tx + \ii\beta_j(x_d - R)}
            \quad\text{for }
            x_d > R
            ,
        \end{equation}
        where $\widehat{(u_{\alpha}\big|_{\Gamma_0^R})}_j$ is the $j$-th Fourier coefficient of the trace. For $\phi \in H^s_\alpha(\Gamma_0^R)$, $ s \in \R$, $j \in \Z^{d-1}$ and $\alpha \in I_{}$, the $j$-th Fourier coefficient $\hat{\phi}_j(\alpha)$ of $\phi$ is defined by
        \begin{equation}
            \hat{\phi}_j(\alpha)
            := \int_{\Gamma_0^R} \phi(\tx) e^{\ii\alpha \cdot \tx} \overline{\psi^j(\tx) }  \dd \tx
            ,
            \quad 
            \text{where}
            \quad 
            \psi^j(\tx) 
            := \frac{1}{(2\pi)^{\nicefrac{(d-1)}{2}}} e^{-\ii j \cdot \tx}
            .
        \end{equation}
        From the radiation condition, we derive
        the bounded quasi-periodic Dirichlet-to-Neumann operator
        $T_\alpha: H_\alpha^{\nicefrac{1}{2}}(\Gamma_0^R) \to H_\alpha^{\nicefrac{-1}{2}}(\Gamma_0^R)$ for $\phi = \sum_{j \in \Z^{d-1}} \hat{\phi}_j(\alpha) e^{-\ii \alpha \cdot \tx} \psi^j(\tx) $ by
        \[
            T_\alpha(\phi)(\tx)
            = \frac{\ii}{(2\pi)^{\nicefrac{(d-1)}{2}}}  \sum_{j \in \Z^{d-1}} \sqrt{k^2 - |\alpha + j |^2} \hat{\phi}_j(\alpha) e^{-\ii (\alpha + j) \cdot \tx}
            .
        \]
        \begin{theorem}\label{thm_equivalence}
            Set $\mathcal{J} := \J$. A function $u \in \tilde{H}^1(\Omega^R)$ solves \Cref{prob_Var1} for $q=0$ if and only if $\mathcal{J} u \in L^2(I_{}; \tilde{H}^1_{\alpha}(\Omega_0^R))$ solves the transformed variational problem
            \begin{align}\label{eq_VarUnterIntegral}
				&\int_{I_{}} \left[ \int_{\Omega_0^R} \nabla_x \mathcal{J} u \cdot \nabla_x \overline{v} - k^2 n_p^2 \mathcal{J} u \overline{v}
				\dd x
				- \int_{\Gamma_0^R} T_\alpha(\mathcal{J} u \big|_{\Gamma_0^R}) \overline{v} \big|_{\Gamma_0^R} \dd S \right] \dd \alpha
				\\
                \nonumber
                &
                =  \int_{I_{}} \int_{\Omega_0^R} (\mathcal{J} f)(\alpha, \cdot) \overline{v} \dd x \dd \alpha
			\end{align}
			for all $v \in L^2(I_{}; \tilde{H}^1_{\alpha}(\Omega_0^R))$. Furthermore, the radiation conditions \eqref{eq_radiationCond} and \eqref{eq_alphaRadiationCond} are equivalent for the corresponding problem.
        \end{theorem}
        
        \begin{proof}
            Set additionally $u_\alpha := (\J u)(\alpha, \cdot) \in \tilde{H}^1_{\alpha}(\Omega_0^R)$ for  $u \in \tilde{H}^1(\Omega^R)$. 
            From \cite{Lechl2016} we know that the transform is an isomorphism between $\tilde{H}^s(\Omega^R)$ and $L^2(I_{}; \tilde{H}^s_{\alpha}(\Omega_0^R))$ for $s \in \R$, that the adjoint operator  $\mathcal{J}^*$ can be identified with the inverse operator $\mathcal{J}^{-1}$, that one can interchange the transform with weak derivation and that the identity $\mathcal{J} (n^2 w) = n^2 \mathcal{J} w$ holds for every $w \in L^2(\Omega^R)$. Applying these properties, we derive the equivalent sesquilinear form for the volume part as follows:
            \begin{align*}
                \int_{\Omega^R} \nabla u \cdot \nabla \overline{v} - k^2 n^2 u \overline{v}  \dd x 
                &= \int_{\Omega^R}   \nabla u \cdot \overline{\mathcal{J}^{-1} (\mathcal{J}  \nabla v)} - k^2 n^2 u  \overline{\mathcal{J}^{-1} (\mathcal{J} v) }
                 \dd x 
                \\
                &= \int_{\Omega^R}   \nabla u \cdot \overline{\mathcal{J}^{*}(\nabla_x \mathcal{J}   v)} - k^2 n^2 u  \overline{\mathcal{J}^{*}(\mathcal{J} v) }
                 \dd x 
                \\
                &=  \int_{I_{}} \int_{\Omega_0^R} \nabla_x  (\mathcal{J}  u) \cdot \nabla_x (\overline{\mathcal{J}   v}) -  k^2 n^2 (\mathcal{J} u) (\overline{\mathcal{J} v} )
                \dd x  \dd \alpha
                \\
                &=  \int_{I_{}}\int_{\Omega_0^R} \left( \nabla_x {u}_{\alpha} \cdot \nabla_x \overline{{v}}_{\alpha} -  k^2 n^2 {u}_{\alpha} \overline{{v}}_{\alpha}  \right)\dd x  \dd \alpha
                .
            \end{align*}
            The right hand side can be treated analogously. Now we have to show the equivalence on the boundary.
            
            Calling $\gamma_{\Gamma^{R}} : \tilde{H}^1(\Omega^R) \to H^{\nicefrac{1}{2}}(\Gamma^R)$ the trace operator on $\tilde{H}^1(\Omega^R)$ and $\gamma_{\Gamma_0^{R}}  : \tilde{H}_\alpha^1(\Omega_0^R) \to {H}_\alpha^{\nicefrac{1}{2}}(\Gamma_0^{R})$ the trace operator on $\tilde{H}^1_{\alpha}(\Omega_0^R)$, we use the identification of the inverse Bloch-Floquet transform with its adjoint operator to get
            \begin{align*}
                \int_{\Gamma^R} T (\gamma_{\Gamma^{R}} u) \gamma_{\Gamma^{R}}\overline{v} \dd S
                = \int_{I_{}} \int_{\Gamma_0^R} (\mathcal{J} T (\gamma_{\Gamma^{R}} u)) (\alpha, \tx) (\overline{\mathcal{J} \gamma_{\Gamma^{R}} v}) (\alpha, \tx)\dd S(\tx) \dd \alpha
                .
            \end{align*}
            It holds the identity $\gamma_{\Gamma_0^{R}} \mathcal{J} u = \mathcal{J} \gamma_{\Gamma^{R}} u$, such that it remains to show that 
            \[
                T_\alpha  (\mathcal{J} u)(\alpha, \cdot) 
                = (\mathcal{J}  T  u)(\alpha, \cdot)
                \quad\text{for all }
                u \in H^{\nicefrac{1}{2}}(\R^{d-1})
                .
            \]

            We define for every smooth function with compact support $ \phi \in C_0^\infty(\R^{d-1})$ the operator 
            \begin{align*}
                (\tilde{J} \phi)(\alpha, \tx)
                := \sum_{j \in \Z^{d-1}} \phi(\alpha +  j) e^{-\ii \alpha \cdot \tx} \psi^j(\tx)
                ,\
                (\alpha, \tx, R) \in I_{} \times \Gamma_0^R
                ,
            \end{align*}
            which can be written as $\tilde{J} = \mathcal{J} \circ \mathcal{F}^{-1}$, where  $\mathcal{F} $ is the Fourier transform (see \cite{Lechl2016}). This implies, in particular, that $\tilde{J}$ is an isomorphism between the spaces $L^2_s(\Gamma^R)$ and $L^2(I; H^s_\alpha(\Gamma_0^R))$ for $s \in \R$, where $L^2_s(\Gamma^R)$ is the subspace of $L^2(\Gamma^R)$ functions, for which the norm $||\xi \mapsto (1+|\xi|^2)^{s/2} w(\xi) ||_{L^2(\R^{d-1})}$ is finite.
            Putting the operator $\tilde{J}$ into the definition of the Dirichlet-to-Neumann operator $T $, we conclude that
            \begin{equation*}
                \mathcal{J} \circ T   u
                = \mathcal{J} \circ \mathcal{F}^{-1}\left( \xi \mapsto  \ii \sqrt{k^2 - |\xi|^2} \mathcal{F}{ (\gamma_{\Gamma^{R}} u)}(\xi) \right)
                = \tilde{J} \left( \xi \mapsto  \ii \sqrt{k^2 - |\xi|^2} \mathcal{F}{ (\gamma_{\Gamma^{R}} u)}(\xi) \right)
                .
            \end{equation*}
            Since it holds $(\xi \mapsto  \ii \sqrt{k^2 - |\xi|^2} \mathcal{F}{ (\gamma_{\Gamma^{R}} u)}(\xi) ) \in L^2_{s-1}(\R^{d-1})$ and $\mathcal{F}{ (\gamma_{\Gamma^{R}} u)} (\alpha +  j )  = \widehat{(\gamma_{\Gamma_0^{R}} {u}_\alpha)}_j(\alpha)$, we finally obtain the claimed identification.
                
            For the radiation condition, one can use the same identity $\widehat{(\gamma_{\Gamma_0^R} {u}_\alpha)}_j(\alpha)  = \mathcal{F} (\gamma_{\Gamma^R} u) (\alpha + j )$ to directly calculate the equivalence of the radiation conditions.
        \end{proof}

        \begin{theorem}\label{thm_uniquesolution_per}
            If the \Cref{assumption_absorption} holds, then the variational problem \eqref{eq_VarUnterIntegral} is uniquely solvable.
        \end{theorem}
        We split the proof into three lemmas.
        
        {
            \begin{lemma}\label{lemma_pktExistence}
                For all $\alpha \in \overline{I_{}}$, there exists a unique solution ${w}_\alpha \in \tilde{H}_\alpha^1(\Omega_0^R)$ to the variational problem
                \begin{align}\label{eq_VarPointwise}
                    a_\alpha(w_\alpha, v) 
                    &:=\int_{\Omega_0^R} \nabla_x w_\alpha \cdot \nabla_x \overline{v} - k^2 n_p^2 w_\alpha \overline{v}
                    \dd x
                    - \int_{\Gamma_0^R} T_\alpha(w_\alpha \big|_{\Gamma_0^R}) \overline{v} \big|_{\Gamma_0^R} \dd S
                    \nonumber
                    \\
                    &=  \int_{\Omega_0^R} (\J f)(\alpha, \cdot) \overline{v} \dd x
                \end{align}
                for every $\overline{v} \in \tilde{H}_\alpha^1(\Omega_0^R)$.
            \end{lemma}
            \begin{proof}
                Let $w_\alpha$ be in $ \tilde{H}_\alpha^1(\Omega_0^R)$ for a fixed $\alpha \in \overline{I_{}}$ and set $\widehat{{w}}_j := \widehat{({w}_\alpha  \big|_{\Gamma_0^R})}_j(\alpha)$, then it holds for $\beta_j := \sqrt{k^2 - |\alpha + j |^2}$
                \begin{align*}
                    - \Re\left(  \int_{\Gamma_0^{ R}} T_\alpha({w}_\alpha \big|_{\Gamma_0^R})  \overline{{w}}_\alpha \big|_{\Gamma_0^R} \dd S  \right)
                    =  \sum_{\substack{j \in \Z^{d-1} \\ |\alpha_j| > k}} |\beta_j|\ |\widehat{{w}}_j|^2
                    \geq 0
                    ,
                \end{align*}
                which implies
                \begin{equation*}
                    \Re  a_\alpha({w}_\alpha, {w}_\alpha)
                    \geq ||{w}_\alpha||^2_{H_{\alpha}^1(\Omega_0^R)} - ||1-k^2 n^2||_{L^2(\Omega_0^{R_0})} ||{w}_\alpha||^2_{L^2(\Omega_0^R)}
                    .
                \end{equation*}
                Thus, the sesquilinear form fulfills the G\aa rding inequality. In the case of $||1-k^2 n^2||_{L^2(\Omega_0^{R_0})} = 0$, the problem is solvable by the theorem of Lax and Milgram. Because of the compact embedding of $\tilde{H}_\alpha^1(\Omega_0^R)$ into $L^2(\Omega_0^R)$, the equation corresponds to a Fredholm operator of index zero. 
                Consequently, by showing the injectivity, we obtain the unique existence of the solution.
            
                For the boundary integral, it holds the inequality
                \begin{equation}\label{eq_ImRand}
                    \Im \left(  \int_{\Gamma_0^{R}} T_\alpha ({w}_\alpha \big|_{\Gamma_0^R}) \overline{w}_\alpha \big|_{\Gamma_0^R} \dd S  \right)
                    =  \sum_{\substack{j \in \Z^{d-1} \\ |\alpha_j| < k}} |\beta_j|\ |\widehat{{w}}_j|^2
                    \geq 0
                    .
                \end{equation}
                Since we assume $\Im(n_p^2) \geq 0$ in $\Omega_0^R$ and $\Im(n_p^2)> 0$ on an open ball of $\Omega_0^R$, we derive for $(\J f)(\alpha, \cdot) = 0$
                \begin{equation*}
                    0
                    = \Im \left( \int_{\Omega_0^R} k^2 n_p^2 |w_\alpha|^2  \dd x + \int_{\Gamma_0^{R}} T_\alpha( w_\alpha) \overline{w}_\alpha \dd S  \right)
                    \geq   \int_{\Omega_0^R} k^2 \Im n_p^2 |w_\alpha|^2  \dd x 
                    \geq 0
                    .
                \end{equation*}
                We conclude that $w_\alpha$ vanishes on the open set, where $\Im n_p^2 > 0$, and the theorem of unique continuation implies that $w_\alpha$ is equal to zero everywhere in $\Omega_0^R$.
            \end{proof}

            Using the same argumentation of the second part, we also get uniqueness for the integrated form \eqref{eq_VarUnterIntegral}.
            \begin{corollary}[]{}
                Every solution to the variational problem \eqref{eq_VarUnterIntegral} is unique.
            \end{corollary}
            \begin{proof}
                Because of the inequality \eqref{eq_ImRand}, we have for $(\J f) = 0$ and the corresponding solution $w \in L^2(I_{}; \tilde{H}^1_{\alpha}(\Omega_0^R))$ the estimation
                \begin{equation*}
                    0
                    = \Im \left(\int_{I_{}} \int_{\Omega_0^R} k^2 n_p^2 |w_\alpha|^2  \dd x + \int_{\Gamma_0^{R}} T_\alpha(w_\alpha) \overline{w}_\alpha \dd S  \dd \alpha \right)
                    \geq 0
                    ,
                \end{equation*}
                where $w_\alpha := w(\alpha, \cdot)$. This implies that $w_\alpha$ vanishes on an open ball for almost every $\alpha \in I_{}$. Since $w_\alpha$ solves the Helmholtz equation almost everywhere in $I_{}$, the unique continuation property implies that $w_\alpha$ vanishes everywhere w.r.t.\ to $x$ and almost everywhere in $I_{}$.
            \end{proof}

            Now, we prove the connection between the pointwise variational problem and the integrated form.
            \begin{lemma}
                The variational problem \eqref{eq_VarUnterIntegral} is uniquely solvable.
            \end{lemma}
            \begin{proof}
                If we define the function $w$, $w(\alpha, \cdot) := {u}_\alpha$, where $u_\alpha$ solve \eqref{eq_VarPointwise} for all $\alpha \in \overline{I_{}}$, \Cref{lemma_pktExistence} implies that $w$ solves the problem \eqref{eq_VarUnterIntegral}.
                What still needs to be checked, is that $w$ lies in $L^2(I_{}; \tilde{H}^1_{\alpha}(\Omega_0^R))$. 
                For that, we show that the solution operator $L_\alpha$ for the problem \eqref{eq_VarPointwise} is uniformly continuous in $\overline{I_{}}$.

                At first, we consider the continuity of the sesquilinear form \eqref{lemma_pktExistence}. For every function ${v}_\alpha \in \tilde{H}^1_{\alpha}(\Omega^R)$, there exists a function ${v}^p_\alpha \in \tilde{H}^1_p(\Omega_0^R)$, where $\tilde{H}^1_p(\Omega_0^R)$ is the space with $\alpha = 0$, such that ${v}_\alpha (\tx, x_d) = e^{i \alpha \cdot \tx} {v}^p_\alpha(\tx, x_d)$. 
                Moreover, the norms of the two functions are equal: $||{v}_\alpha||_{\tilde{H}^1_{\alpha}(\Omega_0^R)}=||{v}^p_\alpha||_{\tilde{H}^1_p(\Omega_0^R)}$. Now, we choose ${v}_\alpha$, ${u}_\alpha \in \tilde{H}^1_{\alpha}(\Omega^R)$ and ${v}^p_\alpha$, ${u}^p_\alpha \in \tilde{H}^1_p(\Omega^R)$ as described, and plugging them into the sesquilinear form \eqref{lemma_pktExistence}  yields
                \begin{align*}
                    {b}_\alpha({u}^p_\alpha, {v}^p_\alpha)
                    :=
                    &
                    \int_{\Omega_0^R} \left( \nabla {u}^p_\alpha \cdot \nabla \overline{{v}}^p_\alpha + (|\alpha|^2  -  k^2 n_p^2) {u}^p_\alpha \overline{{v}}^p_\alpha  +  \ii {u}^p_\alpha \alpha \cdot \nabla_{\tx} \overline{{v}}^p_\alpha - \ii \overline{{v}}^p_\alpha \alpha \cdot \nabla_{\tx} {u}^p_\alpha \right) \dd x 
                    \\
                    &\quad \quad \quad 
                    -\int_{\Gamma_0^{R}} S_\alpha( {u}^p_\alpha\big|_{\Gamma_0^R})  \overline{{v}}^p_\alpha\big|_{\Gamma_0^R} \dd S
                    ,
                \end{align*}
                where
                \begin{align*}
                    \int_{\Gamma_0^{R}} S_\alpha( {u}^p_\alpha\big|_{\Gamma_0^R})  \overline{{v}}^p_\alpha\big|_{\Gamma_0^R} \dd S
                    := \sum_{j \in \Z^2} \ii \beta_j \widehat{( {u}^p_\alpha\big|_{\Gamma_0^R})}_j (\alpha)\overline{\widehat{( {v}^p_\alpha \big|_{\Gamma_0^R})}_j}(\alpha)
                    .
                \end{align*}
                In contrary to  $T_\alpha$, the operator $S_\alpha$ only depends on $\alpha$ by the coefficients $\beta_j(\alpha) := \sqrt{k^2 - |\alpha + j|^2}$.
                
                Fix $\varepsilon > 0$, $\alpha \in \R^{d-1}$ and $\alpha_\varepsilon \in \R^{d-1}$, where $|\alpha_\varepsilon - \alpha| < \varepsilon$, such that for every $u$, $v \in \tilde{H}^1_p(\Omega_0^R)$ it holds
                \begin{align*}
                    |{b}_{\alpha_\varepsilon} (u,v) - {b}_{\alpha}(u,v) |
                    &\leq \left( | |\alpha_\varepsilon|^2 - |\alpha|^2 | +2|\alpha_\varepsilon - \alpha| \right) ||u ||_{\tilde{H}^1_p(\Omega_0^R)} ||v||_{\tilde{H}^1_p(\Omega_0^R)} 
                    \\
                    &\ \ \ \ + \sum_{j \in \Z^{d-1}} |\beta_j(\alpha_\varepsilon) - \beta_j(\alpha)|  \left|(\widehat{ {u}\big|_{\Gamma_0^R}})_j \overline{(\widehat{ {v}\big|_{\Gamma_0^R}})_j}\right|
                    \\
                    &\leq \left( | |\alpha_\varepsilon|^2 - |\alpha|^2 | +2|\alpha_\varepsilon - \alpha| \right) ||u ||_{\tilde{H}^1_p(\Omega_0^R)} ||v||_{\tilde{H}^1_p(\Omega_0^R)} 
                    \\
                    &\ \ \ \ +  C(\alpha_\varepsilon,\alpha ) ||u ||_{\tilde{H}^1_p(\Omega_0^R)} ||v||_{\tilde{H}^1_p(\Omega_0^R)}
                    ,
                \end{align*}
                where
                \begin{align*}
                    C(\alpha_\varepsilon,\alpha )
                    := c_{\operatorname{trace}}\sup_{j \in \Z^{d-1}} \frac{|({k^2 - | j + \alpha_\varepsilon |^2})^{1/2} - ({k^2 - | j + \alpha |^2})^{1/2}|}{(1+|j|^2)^{1/2}}
                    .
                \end{align*}
                For $j = 0$ and for $j \in \Z^{d-1}$ with $k^2 = |j + \alpha |^2$, the fraction $C(\alpha_\varepsilon,\alpha )$ is continuous in $I_{}$. For other $j \in  \Z^{d-1}$, it holds
                \begin{align*}
                    &
                    \frac{|({k^2 - | j + \alpha_\varepsilon |^2})^{1/2} - ({k^2 - | j + \alpha |^2})^{1/2}|}{(1+|j|^2)^{1/2}}
                    \\
                    &
                    =  \frac{| | j + \alpha_\varepsilon |^2 -  | j + \alpha |^2 |}{(1+|j|^2)^{1/2}|({k^2 - | j + \alpha_\varepsilon |^2})^{1/2} + ({k^2 - | j + \alpha |^2})^{1/2}|}
                    .
                \end{align*}
                For every $j\in \Z^{d-1}$ with $k^2 \neq | j + \alpha |^2$, the value $\beta_j(\alpha)$ is contained either in $(-\infty, 0)$, or $\ii(0, \infty)$, and fulfills $|\beta_j(\alpha)| > \delta$ for a small constant $\delta > 0$ independent of $j$. It follows
                \begin{align*}
                    &|({k^2 - | j + \alpha_\varepsilon |^2})^{1/2} + ({k^2 - | j + \alpha |^2})^{1/2}|
                    \geq \delta
                \end{align*}
                for all $\alpha_\varepsilon \in \overline{I_{}}$.
                Thus, it holds the estimation
                \begin{align*}
                    C(\alpha_\varepsilon,\alpha )
                    \leq  \sup_{j \in \Z^{d-1}} \frac{c  }{|j|\delta} \left| \sum^2_{i=1} (\alpha_\varepsilon - \alpha)_i (\alpha_\varepsilon + \alpha + 2   j )_i  \right|
                    \to 0 
                    \quad\text{for }
                    \alpha_\varepsilon \to \alpha
                    ,
                \end{align*}
                which implicates that the operator $\alpha \mapsto {b}_\alpha$
                is continuous from $\overline{I_{}}$ into  $\mathcal{L}(\tilde{H}^1_{\alpha}(\Omega_0^R); \mathcal{L}(\tilde{H}^1_{\alpha}(\Omega_0^R); \CC))$.
                
                Since the sesquilinear form $a_\alpha$ is equivalent to $b_\alpha$, and since the norms of the spaces are equivalent, the sesquilinear form $a_\alpha$ is also continuous. Applying the Neumann series argument, we obtain that the solution operator $\alpha \mapsto L_\alpha$ is continuous on the compact set $\overline{I_{}}$, and thus, bounded by a constant $C$ independent of $\alpha$. In particular, the function $w$ lies in  ${L^2(I_{};\tilde{H}^1_{\alpha}(\Omega_0^R))}$, since
                \begin{align*}
                    ||{w}||^2_{L^2(I_{}; \tilde{H}^1_{\alpha}(\Omega_0^R))}
                    &= \int_{I_{}} ||{u}_\alpha||^2_{\tilde{H}^1_{\alpha}(\Omega_0^R)} \dd \alpha
                    \leq \int_{I_{}} ||L_\alpha||^2\, ||(\J f)(\alpha, \cdot)||^2_{ L^2(\Omega_0^R)}  \dd \alpha
                    \\
                    &\leq C^2  ||\J f||^2_{L^2(I_{} \times \Omega_0^R)}
                    .
                \end{align*}
            \end{proof}
        }
    }
	\subsection{Locally perturbed periodic inhomogeneous layer scattering}\label{Sec23}
    {
		Combining \Cref{thm_uniquesolution_per} and \Cref{thm_equivalence}, we obtain the unique existence of a solution for the unperturbed scattering problem. 
        Now we consider the case that the perturbation $q \in L^\infty(\Omega_0^{R_0})$ is not vanishing.
        With the results from the subsection above, we are able to prove \Cref{thm_whole}.
        \begin{proof}[Proof of \Cref{thm_whole}]
            The sesquilinear form ${l} : \tilde{H}^1(\Omega^R) \times \tilde{H}^1(\Omega^R) \to \CC$,
            \[
                {l}(u, v) 
                := \int_{\Omega^R}  k^2 q u  \overline{v}  \dd x 
                ,
            \]
            is a compact perturbation, since $q$ vanishes outside of $\Omega_0^{R_0}$. As we showed earlier, the unperturbed problem is uniquely solvable, such that the variational formulation \ref{prob_Var1} corresponds to a Fredholm operator of index zero. Thus, we have to show uniqueness, which can be proven by using the same argumentation as in \Cref{lemma_pktExistence}, if \Cref{assumption_absorption} holds, since for the solution $w = \J u$ to $f = 0$ it holds
            \begin{align*}
                0
                &= \Im \left(
                \int_{I_{}} \int_{\Omega_0^R} k^2 n_p^2 |w|^2  \dd x 
                + \int_{I_{}} \int_{\Omega_0^R} k^2 q u \overline{w} \dd x 
                \dd \alpha
                \right)
                \\
                &=
                \int_{I_{}} \int_{\Omega_0^R} k^2 \Im (n_p^2) |w|^2  \dd x \dd \alpha
                + \int_{\Omega_0^R} k^2 \Im (q) |u|^2  \dd x 
                \\
                &\geq 0
                .
            \end{align*}
        \end{proof}
		
        As the last point, we show the regularity of the quasi-periodic solutions w.r.t.\ parameter $\alpha$. We will use this result for the implementation of the algorithm, since we can improve the convergence rate of the inverse transform with it.        
        As the first step and defining the set
        \[
            \mathcal{A}
            := \{
                    \alpha \in \overline{I_{}} : |\alpha + j| = k
                    \text{ for some }
                    j \in \Z^{d-1}
                \}
            ,
        \]
        one can show the regularity result for the unperturbed case by applying the Neumann series argument.
        \begin{theorem}\label{thm_regularity_per}
            If the right hand side $\J f$ is analytical in $\alpha \in \overline{I_{}}$, then the map $\alpha \mapsto u_\alpha$, where $u_\alpha$ solves the quasiperiodic problem \eqref{eq_VarPointwise}, is analytically in $\overline{I_{}} \setminus \mathcal{A}$, and for any $\hat{\alpha} \in \mathcal{A}$, there exists a $j_0 \in \Z^{d-1}$ and a neighborhood $U(\hat{\alpha}) \subseteq \R^{d-1}$ of $\hat{\alpha}$, such that the function can be decomposed into two analytical functions $u^{(1)}_\alpha$ and $u^{(2)}_\alpha$ in the form
            \[
                u_\alpha
                = u^{(1)}_\alpha + \sqrt{k^2 - |\alpha + j_0|^2} u^{(2)}_\alpha
                \quad\text{for }
                \alpha \in U(\hat{\alpha})
                .
            \]
        \end{theorem}
        \begin{proof}
            This can be showed analogously to \cite[Theorem~a]{Kirsch1993}, which treats the case of the quasi-periodic scattering problem with sound-soft boundary conditions. Loosely speaking, one can split the differential operator $D_\alpha$ into $D^{(1)}_\alpha + \sqrt{k^2 - |\alpha + j_0|^2} D^{(2)}_\alpha$, where both operators $D^{(1)}_\alpha$ and $D^{(2)}_\alpha$ are analytical in $\alpha$. Since $\sqrt{k^2 - |\alpha + j_0|^2} \to 0$ for $|\alpha + j_0|^2 \to k^2$, the Neumann series argument implies that the inverse of $D_\alpha$ can be decomposed in the same way.
        \end{proof}
        
		Since the compact perturbation of the sesquilinear form is independent of $\alpha$, one gets an analogous decomposition result to \Cref{thm_regularity_per}.
        \begin{theorem}\label{thm_regularity_perturbed}
        	If the right hand side $\J f$ is analytical in $\alpha \in \overline{I_{}}$, the function $u_\alpha = \J u(\alpha, \cdot)$, where $u \in \tilde{H}^1(\Omega^R)$ solves the (perturbed) variational problem \ref{prob_Var1}, is analytically dependent on $\alpha \in \overline{I_{}} \setminus \mathcal{A}$. For any $\hat{\alpha} \in \mathcal{A}$, one can find a $j_0 \in \Z^{d-1}$, a neighborhood $U(\hat{\alpha})$ of $\hat{\alpha}$, and two analytical functions $u^{(1)}_\alpha$ and $u^{(2)}_\alpha$, such that $u_\alpha$ can be written as
			\begin{equation}\label{eq_pertReg_represent}
				u_\alpha
				= u^{(1)}_\alpha + \sqrt{k^2 - |\alpha + j_0|^2} u^{(2)}_\alpha
				\quad\text{for }
				\alpha \in U(\hat{\alpha})
				.
			\end{equation}
		\end{theorem}
		\begin{proof}
            Let $K_q \in \mathcal{L}(L^2(I_{}; \tilde{H}^1_{\alpha} (\Omega_0^R)))$
            be the Riesz representation of $(w \mapsto k^2q \J^{-1} w) \in \mathcal{L}(L^2(I_{}; \tilde{H}^1_{\alpha} (\Omega_0^R)), L^2(I_{}; \tilde{H}^1_{\alpha} (\Omega_0^R))')$,
            \[
                (K_q w, v)_{L^2(I_{}; \tilde{H}^1_{\alpha} (\Omega_0^R))}
                =    k^2 \int_{I_{}} \int_{\Omega^R_0} q \J^{-1} w \overline{v} \dd x \dd\alpha
                .
            \]
            The operator $K_q$ maps functions from $L^2(I_{}; \tilde{H}^1_{\alpha} (\Omega_0^R))$ to functions, which are independent of $\alpha$, and thus, in particular, analytical in $\alpha$.
            
            Let $w  = \J u\in L^2(I_{}; \tilde{H}^1_{\alpha} (\Omega_0^R))$ be the solution to the perturbed variational problem \ref{prob_Var1}, and $A \in \mathcal{L}(L^2(I_{}; \tilde{H}^1_{\alpha} (\Omega_0^R)))$ the Riesz representation of the unperturbed invertible differential operator for $q=0$. If we call the Riesz representation of the right hand side  as $\tilde{f}$, then it holds
            \[
                w
                = A^{-1} \tilde{f} + A^{-1}  {K}_q w
                \text{ in }
                L^2(I_{}; \tilde{H}^1_{\alpha} (\Omega_0^R))
                .
            \]
            Since the right hand side $\tilde{f}$ and the function ${K}_q w$ are analytical in $\alpha$, \Cref{thm_regularity_per} implies that $w$ can be represented in the form of \eqref{eq_pertReg_represent}.
		\end{proof}
        
        In \cite{Zhang2018} you can find comparable results for the sound-soft obstacle scattering problem and a detailed description, how to use the regularity to get a better convergence of the discretized inverse Bloch-Floquet transform.
         
        \begin{remark}
            One can extend the regularity result in \Cref{thm_regularity_perturbed} easily for the case that the right hand site $f_\alpha$ can be decomposed in the same way as $f_\alpha = f_\alpha^{(1)} + \sqrt{k^2 - |\alpha + j_0|^2} f_\alpha^{(2)}$, where $f_\alpha^{(1)}$ and $f_\alpha^{(2)}$ are analytical in $\alpha$. 
        \end{remark}
    }

	\section{The Inverse Problem}\label{Sec3}
    {
        In this section we consider the inverse problem of reconstructing the perturbation. 
        For that, we will consider the operator, which maps the perturbation to the solution operator for every right hand side $f \in L^2(\Omega^R)$ with the support in one periodic cell $\Omega_0^{R_0}$.
        
        At first, we will define the domain of definition for the measurement operators. For that, notice that in $d=2,3$ the space $\tilde{H}^1(\Omega^R_0)$ is continuously embedded in $L^4(\Omega_0^R)$. Thus, for a $q_0 \in L^\infty(\Omega^R_0)$, $\Im q_0 \geq 0$, and $q \in L^2(\Omega_0^{R_0})$, such that $||q-q_0||_{L^2(\Omega_0^{R_0})} < \delta$ with $\delta > 0$ small, for every $u$ and $v \in \tilde{H}^1(\Omega^R)$ it holds the estimation
        \begin{align*}
            |a_q(u, v)|
            &\leq |a_{q_0}(u, v)| + k^2 ||q-q_0||_{L^2(\Omega^R_0)} ||u||_{L^4(\Omega^R_0)}||v||_{L^4(\Omega^R_0)}
            \\
            &\leq C(q_0)||u||_{H^1(\Omega^R)}||v||_{H^1(\Omega^R)} + C ||q-q_0||_{L^2(\Omega^R_0)} ||u||_{H^1(\Omega^R_0)}||v||_{H^1(\Omega^R_0)}
            \\
            &\leq (C(q_0) + C \delta)||u||_{H^1(\Omega^R)}||v||_{H^1(\Omega^R)}
            .
        \end{align*}
        Consequently, for a small $\delta(q_0)$, the sesquilinear form $a_q$ is a small perturbation of $a_{q_0}$, and the Neumann series argument guaranties the invertibility of the differential operator for perturbation $q$ of $n^2_p$.
        Since we need an open set as the domain of definition of the measurement operators, and the inversion methods for inverse problems depend on Hilbert spaces, we define the domain of definition $Q$ as 
        \[
            Q :=
                \bigcup_{\substack{q_0 \in L^\infty(\Omega_0^R) \\ \Im q_0 \geq 0}}
                    B_{\delta(q_0)}(q_0)
            \subseteq L^2(\Omega_0^{R_0})
                    ,
        \]
        where $B_{\delta(q_0)}(q_0) \subseteq L^2(\Omega_0^R)$ is an open ball in $L^2(\Omega^R_0)$ around a perturbation $q_0 \in L^\infty(\Omega_0^R)$ with the radius $\delta(q_0)$ depending on $q_0$. Because of the Neumann series and the continuity of the sesquilinear form, the solution operator is well-defined for every $q \in Q$.
        \begin{definition}
            Consider the linear and bounded operator $\Lambda_{q_0} : L^2(\Omega_0^{R_0}) \to L^2(\Omega_0^R)$, which maps a right hand side $f \in L^2(\Omega_0^{R_0})$ to the restriction of the solution $u_{q_0} \in \tilde{H}^1(\Omega^R)$ of \Cref{prob_Var1} with $q = q_0$ to $\Omega_0^R$. We define the first measurement operator as 
            \begin{align*}
            	\Lambda : Q \subseteq L^2(\Omega_0^R) &\to \mathcal{L}(L^2(\Omega_0^{R_0}),  L^2(\Omega_0^R))
            	,
            	\\
            	q &\mapsto \Lambda_{q},
           	\end{align*}
           	mapping the perturbation $q \in Q$ to the operator $\Lambda_{q}$.
        \end{definition}
        
        \begin{definition}
  			Let $\tilde{\Lambda}_{q_0} : L^2(\Omega^{R_0}) \to \tilde{H}^1(\Omega_0^R)$ be the operator from above with codomain $\tilde{H}^1(\Omega_0^R)$ and let $\gamma_{\Gamma_0^R} : \tilde{H}^1(\Omega^R) \to H^{\nicefrac{1}{2}}(\Gamma_0^R)$ be the trace operator, restricted to $\Gamma_0^R$. We define the second measurement operator 
            \begin{align*}
				\mathcal{S} : Q &\to \mathcal{L}(L^2(\Omega_0^{R_0}),  L^2(\Gamma_0^R))
				,
            	\\
				q &\mapsto \gamma_{\Gamma_0^R} \circ \tilde{\Lambda}_{q}
				,
           	\end{align*}
			which only measures the scattered field on one periodic cell of the upper boundary.
        \end{definition}
    }
	\subsection{Uniqueness of the inverse problem}
    {
        In this section, we will proof the injectivity of both operators $\Lambda$ and $\mathcal{S}$.
        \begin{theorem}
            Consider two perturbations $q_1$ and $q_2 \in Q$. Then it holds:
            \[
                \text{ If } 
                \Lambda(q_1) 
                = \Lambda_{q_1} 
                = \Lambda_{q_2} 
                = \Lambda(q_2)
                ,\ 
                \text{ then }
                q_1 = q_2
                .
	        \]
		\end{theorem}
        \begin{proof}
            For a fixed right hand side $f \in L^2(\Omega_0^{R_0})$, we have two solutions $u_1$ for the variational problem \ref{prob_Var1} with $n^2 =n^2_p + q_1$ and $u_2$ the solution to the \Cref{prob_Var1} with $n^2 =n^2_p + q_2$. 
            Since $u_1$ equals $u_2$ on the set $\Omega_0^R$, the function $w := u_1 - u_2 \in \tilde{H}^1(\Omega^R)$ solves the problem
            \begin{equation*}
                \int_{\Omega^R} \nabla w \cdot \nabla \overline{v} - k^2 n_p^2 w \overline{v}  - k^2 q_1 w\overline{v} \dd x - \int_{\Gamma^R} T w  \overline{v} \dd S
                = \int_{\Omega^R} - k^2(q_1 - q_2) u_2\overline{v} \dd x
                ,
            \end{equation*}
            and vanishes, in especially, on $\Omega_0^R \setminus \Omega^{R_0}_0 \neq \emptyset$. Applying the theorem of unique continuation, it follows that $u_1 = u_2$ on $\Omega^R \setminus \Omega_0^R$, and consequently, the functions are identical on the whole domain $\Omega^R$.
            Thus, for every $v \in \tilde{H}^1(\Omega^R)$ it holds
            \begin{equation*}
                \int_{\Omega^R} (q_1 - q_2) u_2\overline{v} \dd x
                = 0
                ,
            \end{equation*}
            and the lemma of fundamental calculus implies $(q_1 - q_2) u_2 = 0$ almost everywhere. Since we can choose an arbitrary function $u_2 \in C^\infty_0(\Omega_0^{R_0})$, we conclude the identity $q_1 = q_2$ in $\R^d$.
        \end{proof}
        {
            In the case of $d=3$, and additional regularity of the parameter $n_p^2$ and $q$, we can moreover prove injectivity of the operator $\mathcal{S}$, at least, if the whole trace on $\Gamma^R$ is given as data instead of data on one periodic cell.
            For that, we utilize the so called complex geometrical optics. The following proposition is adapted from 
            \cite[Proposition~3.2]{IsakovLaiWang2016} (see also \cite{SylvesterUhlmann1988}).
            
            \begin{proposition}\label{thm_cgo}
                Let $D \subseteq \R^3$ be a bounded domain with Lipschitz boundary $\partial D$, $\xi \in \CC^3$ satisfying $\xi \cdot \xi = 0$ and $\rho \in H^2(D)$.
                Then there exist constants $C_0$ and $C_1$ depending on $D$, such that for $|\xi| > C_0||\rho||_{H^2(D)}$ there exists a solution $u$ of the form
                \begin{equation}\label{eq_cgo_form}
                    u(x)
                    = e^{\xi \cdot x} (1 + \psi_{\xi, \rho}(x))
                    ,
                \end{equation}
                which solves the equation
                \begin{equation*}
                    \Delta u + \rho u
                    = 0
                    \text{ in }
                    D
                \end{equation*}
                and satisfies
                \[
                    \psi_{\xi, \rho} \in H^2(D)
                    ,
                    \ 
                    ||\psi_{\xi, \rho}||_{H^2(D)}
                    \leq \frac{C_1}{|\xi|} ||\rho||_{H^2(D)}
                    .
                \]
            \end{proposition}
            
            \begin{theorem}
                Consider for $d=3$ two perturbations $q_1$ and $q_2 \in Q \cap C^2(\R^3_+)$ with compact support in $\R^3_+$, and assume $n_p^2 \in C^2_p(\Omega_0^R)$. If we call the solution operator $\tilde{\Lambda}_{q} : L^2(\Omega_0^{R_0}) \to \tilde{H}^1(\Omega^R)$ for $q \in Q$, and define $\tilde{\mathcal{S}}: Q \to \mathcal{L}(L^2(\Omega^{R_0}), L^2(\Gamma^R))$, $q \mapsto \gamma_{\Gamma^R} \circ \tilde{\Lambda}_{q}$, where $\gamma_{\Gamma^R}$ is the trace operator, then it holds:
                \[
                    \text{If }
                    \tilde{\mathcal{S}}(q_1) 
                    = \gamma_{\Gamma^R} \circ \tilde{\Lambda}_{q_1}
                    =  \gamma_{\Gamma^R} \circ \tilde{\Lambda}_{q_2}
                    = \tilde{\mathcal{S}}(q_2)
                    ,
                    \text{ then }
                    q_1 
                    = q_2
                    .
            \]
            \end{theorem}
            \begin{proof}
                For a fixed $f \in L^2(\Omega_0^{R_0})$, the operators $\tilde{\mathcal{S}}(q_1)$ and $\tilde{\mathcal{S}}(q_2)$ map the right hand side $f$ to the traces of the solutions $u_1$ and $u_2$ of \Cref{prob_Var1} with $n^2 = n^2_p + q_1$, or, $n^2 = n^2_p + q_2$, respectively. The functions $u_1$ and $u_2$ coincide on $\Gamma^R$ and, of course, on $\Gamma^0$. The difference $w := u_1 - u_2$ satisfies the equation
                \begin{align*}
                    \Delta w + k^2 (n_p^2 + q_1) w
                    &= \left[ \Delta u_1 + k^2 (n_p^2 + q_1) u_1 \right]
                      - \left[ \Delta u_2 + k^2 (n_p^2 + q_2) u_2 \right]
                      + k^2 (q_2 - q_1) u_2
                    \\
                    &= k^2 (q_2 - q_1) u_2
                \end{align*}
                with homogeneous Dirichlet boundary conditions.
                Since the upper trace determines the extension by the radiation condition, we can conclude that $w$ vanishes on an open set for some $R' > R > R_0 > 0$. The unique continuation theorem implies that the function $w$ vanishes on the biggest connected subset $D^c \subseteq \{q_1 = q_2 \} \subseteq \R^3$, which includes the boundary $\Gamma^R \cup \Gamma^0$. Hence $u_1 = u_2$ on $D^c$.
                Putting the functions into the sesquilinear form, we obtain
                \begin{align*}
                    \int_{\Gamma^R} T u_1 u_2  \dd S
                    &= \int_{\Omega^R} \nabla u_1 \cdot \nabla u_2 - k^2 (n_p^2 + q_1) u_1 u_2 \dd x
                        - \int_{\Omega^R} f u_2 \dd x
                    \\
                    &= \int_{\Omega^R} \nabla u_2 \cdot \nabla u_1 - k^2 (n_p^2 + q_2) u_2 u_1 \dd x
                        - \int_{\Omega^R} f u_1 \dd x
                    \\
                    &\quad \quad 
                    +\int_{\Omega^R} k^2 (q_2 - q_1) u_1 u_2 + f(u_1-u_2) \dd x
                    \\
                    &= \int_{\Gamma^R} T u_2 u_1  \dd S
                    +\int_{\Omega^R} k^2 (q_2 - q_1) u_1 u_2 + f(u_1-u_2) \dd x
                    .
                \end{align*}

                Now, we choose two arbitrary right hand sides $f'$ and $g' \in L^2(\Omega_0^{R_0})$ with support in $D^c$, and define, for $\rho(t, x) := t (n_p^2(x) + q_2(x)) + (1-t)(n_p^2(x) + q_1(x))$, $\rho \in C^\infty([0,1] ; L^2(\Omega_0^{R_0}))$, the two solutions to \Cref{prob_Var1} with $n^2 = \rho(t, \cdot)$ as $u^{f'}_{\rho(t, \cdot)}$, or, $v^{g'}_{\rho(t, \cdot)}$, respectively. Then it holds
                \begin{align*}
                    B_{\rho(t, \cdot)} (f', g')
                    &:= \int_{\Omega^R} f' v^{g'}_{\rho(t, \cdot)} + g' u^{f'}_{\rho(t, \cdot)} \dd x
                    \\
                    &= 2 \int_{\Omega^R} \nabla u^{f'}_{\rho(t, \cdot)} \cdot \nabla v^{g'}_{\rho(t, \cdot)} - k^2 \rho u^{f'}_{\rho(t, \cdot)} v^{g'}_{\rho(t, \cdot)} \dd x
                        - 2 \int_{\Gamma^R} T u^{f'}_{\rho(t, \cdot)} v^{g'}_{\rho(t, \cdot)}  \dd S
                    .
                \end{align*}
                Since $u^{f'}_{\rho(0, \cdot)} = u^{f'}_{\rho(1, \cdot)}$ and $v^{g'}_{\rho(0, \cdot)} = v^{g'}_{\rho(1, \cdot)}$ on $D^c$, as we showed earlier, and both function $f'$ and $g'$ are chosen to be zero on $D$, the complement of $D^c$, we have
                \begin{align*}
                    B_{\rho(1, \cdot)} (f', g') - B_{\rho(0, \cdot)} (f', g')
                    = \int_{\Omega^R} f' (v^{g'}_{\rho(1, \cdot)} - v^{g'}_{\rho(0, \cdot)}) + g' (u^{f'}_{\rho(1, \cdot)} - u^{f'}_{\rho(0, \cdot)}) \dd x
                    = 0
                    .
                \end{align*}
                Set $u := u^{f'}_{\rho(t, \cdot)}$, $v := v^{g'}_{\rho(t, \cdot)}$, $u' := \frac{\partial}{\partial t} u$ and $v' := \frac{\partial}{\partial t} v$ for now, then, the upper equation implies
                \begin{align*}
                    0 
                    &= \frac{1}{2} B_{\rho(1, \cdot)} (f', g') - \frac{1}{2}  B_{\rho(0, \cdot)} (f', g')
                    = \frac{1}{2} \int_0^1 \frac{\partial}{\partial t} B_{\rho(t, \cdot)} (f', g') \dd t
                    \\
                    &
                    = \int_0^1 
                        \int_{\Omega^R} \nabla u' \cdot \nabla v + \nabla  u \cdot \nabla v' - k^2 \frac{\partial}{\partial t} \rho u v - k^2 \rho u' v - k^2 \rho u v' \dd x 
                    \\
                    & \quad \quad \quad - \int_{\Gamma^R} T u' v + T v' u  \dd S
                        \dd t
                    .
                \end{align*}
                Differentiating the variational problem \ref{prob_Var1} for $n^2 = \rho$ w.r.t.\ $t$, we conclude that $u'$ and $v'$ solve the problem
                \begin{align*}
                    \Delta w' + k^2 \rho w'
                    &= -k^2 \frac{\partial} { \partial t} \rho w
                    &&\text{in }
                    \Omega^R
                    \\
                    \frac{\partial }{\partial x_3} w'
                    &= T(w')
                    &&\text{on }
                    \Gamma^R
                    \\
                    w'
                    &= 0
                    &&\text{on }
                    \Gamma^0
                    .
                \end{align*}
                Since the derivative of $\rho$ w.r.t.\ $t$ is given by $\frac{\partial }{\partial t} \rho = (q_2 - q_1)$, we obtain
                \begin{equation}\label{eq_intQisZero}
                    \int_{\Omega^R}  (q_2 - q_1) \int_0^1 u^{f'}_{\rho(t, \cdot)}  v^{g'}_{\rho(t, \cdot)} \dd t \dd x
                    = 0
                \end{equation}
                for every $f'$ and $g' \in L^2(\Omega_0^{R_0})$ with $f'\big|_{D} = g'\big|_{D} = 0$.
                
                The assumption that $d=3$, allows us to choose vectors $\xi^i \in \CC^3$, $i = 1,2$ such that the norms $|\xi^i|^2$ are large for $i=1,2$, and both can be decomposed into
                \[
                    \xi^1 
                    = \ii (m + p) + l
                    \text{ and }
                    \xi^2 
                    = \ii (m - p) - l
                \]
                with pairwise orthogonal real vectors $l$, $m$ and $p$, such that $|l|^2 = |m|^2 + |p|^2$.
                Choosing some Lipschitz domain $\tilde{D} \supseteq D$, \Cref{thm_cgo} gives us two functions $u$ and $v \in H^2(\tilde{D})$ of the form \eqref{eq_cgo_form}. 
                Multiplying a cut-off function $\chi \in C^\infty(\R^3)$ to the functions $u$ and $v$, which fulfills $\chi\big|_{D} = 1$ and $\chi\big|_{\R^3 \setminus \tilde{D}} = 0 $, one can see that these functions are solutions to \Cref{prob_Var1} with suitable right hand sides $f'$ and $g'$ supported in $\Omega_0^{R_0} \setminus D$.
                Inserting these two functions into \eqref{eq_intQisZero}, we obtain
                \[
                    \int_{\Omega^R} (q_2 - q_1) e^{2 \ii m \cdot x} \left(1 + \mathcal{O}\left(\frac{1}{|p|}\right)\right) \dd x 
                    = 0
                    .
                \]
                Letting $|p|$ go to infinity, we deduce that the Fourier transform of the function $(q_1 -q_2)$ equals to zero. 
                Consequently, the identity $q_1 = q_2$ holds everywhere in $\R^3$.
            \end{proof}
        }        
	}
    
	\subsection{Fr\'echet differentiability and ill-posedness of the inverse problem}
    {
        In the following, we will apply an inexact Newton-method, called {CG-REGINN} (\cite{Rieder2005}), to reconstruct the shape of the perturbation. For that, we prove differentiability of the measurement operators $\Lambda$ and $\mathcal{S}$ as well as the ill-posedness of the inverse problems.
        \begin{theorem}
			Fix $q \in Q$ and let $u_{f}$ be the solution of \Cref{prob_Var1} for the right hand side $f \in L^2(\Omega_0^{R_0})$. Furthermore, for $h \in L^2(\Omega_0^{R_0})$ let $W_h : L^2(\Omega_0^{R_0}) \to L^2(\Omega_0^R)$, $f \mapsto w_{h, f}\big|_{\Omega_0^R}$, be the operator mapping $f$ to the solution $w_{h, f}$ for the \Cref{prob_Var1} with the right hand side $ k^2 h u_{f}$ replacing $f$, i.e., the function $w_{h, f}$ solves 
            \[
                a_q(w_{h, f}, v)
                = \int_{\Omega_0^R} k^2 h u_{f} \dd x
            \]
            for all $v \in \tilde{H}^1(\Omega^R)$.

			Then, the derivative of $\Lambda$ is given by 
            \begin{align*}
                \Lambda'(q) 
                &\in \mathcal{L}(L^2(\Omega_0^R), \mathcal{L}(L^2(\Omega_0^{R_0}),  L^2(\Omega_0^R)))
                \\
                h &\mapsto W_h
                .
            \end{align*}
		\end{theorem}
		\begin{proof}
            Applying the Riesz theorem, we can reformulate the variational problem \ref{prob_Var1} as
            \begin{align*}
                (B_{q} u, v)_{\tilde{H}^1(\Omega_0^R)}
                = (g, v)_{\tilde{H}^1(\Omega_0^R)}
                \quad\text{for all }
                u,v \in {\tilde{H}^1(\Omega^R)}
                ,
            \end{align*}
            where $B_{q} \in \mathcal{L}(\tilde{H}^1(\Omega_0^R))$ is the Riesz representation of the differential operator of \Cref{prob_Var1} and $g$ the Riesz representation of $f$.
            One can check easily that the sesquilinear form is Fr\'echet differentiable w.r.t.\  the perturbation $q$. It follows that the operator $q \mapsto B_{q}$ has a Fr\'echet derivative $q \mapsto B_{q}' \in \mathcal{L}(L^2(\Omega_0^{R_0}), \mathcal{L}(\tilde{H}^1(\Omega_0^R)) )$.
            
            Since the operator $B_{q}$ is invertible for every $q \in Q$, a corollary of the Neumann series argument implies that the operator $q \mapsto \Lambda_q = B_{q}^{-1}$ is also Fr\'echet differentiable and the linearization can be written as $\Lambda'(q)[h] = -\Lambda_q B'_{q}h \Lambda_q$, which corresponds to the claiming representation.
		\end{proof}
        
        As a consequence, we obtain the differentiability of the measurement operator $\mathcal{S}$.
        \begin{corollary}
			The forward operator $\mathcal{S}$ is Fr\'echet differentiable in $q \in Q$. 
			The derivative is given by $\mathcal{S}'q \in \mathcal{L}(L^2(\Omega_0^{R_0}), \mathcal{L}(L^2(\Omega_0^{R_0}),  L^2(\Gamma_0^R)))$, which maps a function $h \in L^2(\Omega_0^{R_0})$ to the operator $\gamma_{\Gamma_0^R} \circ \tilde{W}_h : L^2(\Omega^{R_0}) \to L^2(\Gamma_0^R)$, where $\tilde{W}_h : L^2(\Omega_0^{R_0}) \to \tilde{H}^1(\Omega_0^R)$ does the same as $W_h$, just mapping to $\tilde{H}^1(\Omega_0^R)$.
        \end{corollary}
	}    	
	{
    
		In the rest of the section, we show that the measurement operator $\Lambda$, and its derivative $\Lambda'$, yields an locally ill-posed inverse problem by proving that the operator satisfies the tangential cone condition. The local ill-posedness of the inverse problem for $\mathcal{S}$ and its derivative $\mathcal{S}'$ can be showed analogously, which we will sketch afterwards.
		
		For a general (non-linear) operator $\Phi : \mathcal{D}(\Phi) \subseteq X \to Y$ between Banach spaces $X$ and $Y$, the operator $\Phi$ is called locally ill-posed in $x^* \in \mathcal{D}(\Phi)$, if for all $r>0$, there exists a sequence $\{x_n\}_{n \in \N} \subseteq B_r(x^*) \cap \mathcal{D}(\Phi)$, such that $||\Phi(x_n) - \Phi(x^*)||_{Y} \to 0$ but $||x_n - x^*||_{X} \not \to 0$ for $n \to \infty$ (\cite[Definition~3.15]{Thomas2012}).
		
		To prove ill-posedness, we will show that $\Lambda : Q \to \mathcal{L}(L^2(\Omega_0^{R_0}),  L^4(\Omega_0^R))$ is locally ill-posed, which implicates that also $\Lambda : Q \to \mathcal{L}(L^2(\Omega_0^{R_0}),  L^2(\Omega_0^R))$ is locally ill-posed. For that, we first show ill-posedness of the linear operator $\Lambda'q_0: L^2(\Omega_0^{R_0})\to \mathcal{L}(L^2(\Omega_0^{R_0}),  L^4(\Omega_0^R))$ for a $q_0 \in Q$, and  conclude afterwards that the inverse problem for $\Lambda$ is locally ill-posed by proving the tangential cone condition.
		
		\begin{lemma}\label{lemma_FrechetIllPosed}
			The operator $\Lambda'(q) : L^2(\Omega_0^{R_0}) \to \mathcal{L}(L^2(\Omega_0^{R_0}), L^4(\Omega_0^R))$ is a compact operator for all $q \in Q$. 
            In particular, the linearized operator equation is locally ill-posed in $\mathcal{L}(L^2(\Omega_0^{R_0}), L^4(\Omega_0^R))$.
		\end{lemma}
		\begin{proof}
            Let $\{ h_j \}_{j \in \N} \subseteq L^2(\Omega_0^{R_0})$ be a weakly convergent sequence, which means that for every functional $\psi \in L^2(\Omega_0^{R_0})'$, it holds $\psi(h_j) \to 0$ for $j \to \infty$.
            Thus, the right hand side $k^2 h_j \Lambda_q f$ converges weakly in $H^{-1}(\Omega^R)$ to zero for every $f \in L^2(\Omega_0^{R_0})$.
            The Sobolev space $H_0^1(\Omega_0^R)$ is compactly embedded in $L^4(\Omega_0^R)$, such that for every $f \in L^2(\Omega_0^{R_0})$ the sequence of solutions $\{W_{h_j} f \}$ converges to zero in $L^4(\Omega_0^R)$ for $j \to \infty$. Applying the theorem of Banach-Steinhaus, we conclude that the sequence of operators  $\{W_{h_j} \}$ converges to zero. Thus, $\Lambda'(q)$ is a compact operator.
		\end{proof}

		\begin{theorem}\label{thm_ill_posedness}
			The inverse problem related to the operator $\Lambda : Q \to \mathcal{L}(L^2(\Omega_0^{R_0}), L^2(\Omega_0^R))$ is locally ill-posed.
		\end{theorem}
		\begin{proof}
			We show that the operator $\Lambda$ satisfies the \emph{tangential cone condition}, which means that for some $q_0 \in Q$ there exist a constant $0 \leq \omega < 1$ and an $r > 0$, such that
			\[
				||\Lambda(q) - \Lambda(q^*) - \Lambda'(q^*)[q-q^*] ||_{\mathcal{L}(L^2(\Omega_0^{R_0}), L^4(\Omega_0^R))}
				\leq \omega ||\Lambda(q) - \Lambda(q^*)||_{\mathcal{L}(L^2(\Omega_0^{R_0}), L^4(\Omega_0^R))}
			\]
			holds for all $q$ and $q^* \in B_r(q_0) \cap Q$. Applying the triangle inequality, one deduces the relation
			
			\[
				1 - \omega
				\leq \frac{||\Lambda'(q^*)[q-q^*] ||_{\mathcal{L}(L^2(\Omega_0^{R_0}), L^4(\Omega_0^R))}}{||\Lambda(q) - \Lambda(q^*)||_{\mathcal{L}(L^2(\Omega_0^{R_0}), L^4(\Omega_0^R))}}
				\leq 1+ \omega
			\]
			for $q \neq q^*$. This, on the other hand, implies, together with \cite[Theorem~4.5]{Gerken2017}, that the local ill-posedness of the inverse problem $\Lambda(q) = \Lambda_q$ follows from the ill-posedness of the inverse problem for the Fr\'echet derivative $\Lambda'(q)$, which we showed in \Cref{lemma_FrechetIllPosed}.
			
			Fix a right hand side $f \in L^2(\Omega_0^{R_0})$ and set $u_q\in \tilde{H}^1(\Omega^R)$ as well as $u_{q^*} \in \tilde{H}^1(\Omega^R)$ as the solutions to \Cref{prob_Var1} for $n^2=n^2_p + q$, or, $n^2=n^2_p + q^*$, respectively. Moreover, let $w_h\in \tilde{H}^1(\Omega^R)$ be the solution to \Cref{prob_Var1} for $n^2=n^2_p + q^*$ and right hand side $k^2 h u_{q^*}$.
			If we define $w$ as $w:=u_q - u_{q^*} - w_h$, then it holds
			\[
				|| \Lambda_{q} f - \Lambda_{q^*} f - W_{(q-q^*)} f||_{L^4(\Omega_0^R)}
				\leq C|| w||_{H^1(\Omega_0^R)}
				.
			\]
			The function $w$ solves the variational problem
			\[
				\int_{\Omega^R} \nabla w \cdot \nabla \overline{v} - k^2 (n^2_p + q^*)w \overline{v} \dd x - \int_{\Gamma^R} T(w \big|_{\Gamma^R}) \overline{v} \dd S
				= \int_{\Omega_0^R} k^2 (q - q^*) (u_{q} - u_{q^*})\overline{v} \dd x
			\]
			for every $v \in \tilde{H}^1(\Omega^R)$.
			Consequently, it holds
			\begin{equation*}
				||w||_{H^1(\Omega_0^R)}
                \leq C k^2 || q - q^*||_{L^2(\Omega_0^{R_0})} || u_q - u_{q^*}||_{L^4(\Omega_0^R)}
                = C k^2 || q - q^*||_{L^2(\Omega_0^{R_0})} || \Lambda_q f - \Lambda_{q^*} f||_{L^4(\Omega_0^R)}
                .
			\end{equation*}
			If the distance $||q - q^*||_{L^2(\Omega_0^{R_0})}$ is small enough, we can set $\omega := C k^2 ||q - q^*||_{{L^2(\Omega_0^{R_0})}} < 1$, wherefrom the tangential cone condition follows, if we take the supremum on both sides:
			\[
				|| \Lambda_{q} - \Lambda_{q^*}  - W_{(q-q^*)} ||_{\mathcal{L}(L^2(\Omega_0^{R_0}), L^4(\Omega_0^R))}
				\leq \omega || \Lambda_q  - \Lambda_{q^*} ||_{\mathcal{L}(L^2(\Omega_0^{R_0}), L^4(\Omega_0^R))}
				.
			\]
            We conclude that the operator $\Lambda : Q \to  \mathcal{L}(L^2(\Omega_0^{R_0}), L^4(\Omega_0^R))$ is locally ill-posed, and thus, $\Lambda : Q \to  \mathcal{L}(L^2(\Omega_0^{R_0}), L^2(\Omega_0^R))$ is also locally ill-posed.
		\end{proof}
        
        The ill-posedness of the inverse problem related to the operator $\mathcal{S}$ can be shown analogously, which we will summarize in the next corollary.
        \begin{corollary}
			The inverse problem related to the operator $\mathcal{S} : Q \to \mathcal{L}(L^2(\Omega_0^{R_0}),  L^2(\Gamma_0^R))$ is locally ill-posed.
        \end{corollary}
        \begin{proof}
            For  $q \in Q$ and $h \in L^2(\Omega_0^{R_0})$ the definition space of the operator $\Lambda_{q}$  and its Fr\'echet derivative $W_{h}$ is $L^2(\Omega_0^{R_0})$, such that both operators map into $H^2(\Omega^R_0)$.
            Thus, one can show analogously to \Cref{lemma_FrechetIllPosed} that the Fr\'echet derivative ${\Lambda}'q$ is a compact operator mapping into $\mathcal{L}(L^2(\Omega_0^{R_0}), H^1(\Omega_0^R))$, and further, one checks analogously to \Cref{thm_ill_posedness} that the tangential cone condition is satisfied for the image space $\mathcal{L}(L^2(\Omega_0^{R_0}), H^1(\Omega_0^R))$. Consequently, the inverse problem related to the operator $\tilde{\Lambda} : Q \to \mathcal{L}(L^2(\Omega_0^{R_0}),  H^1(\Omega_0^R))$ is ill-posed, which implicates that the inverse problem related to $\mathcal{S}$ is also ill-posed, since $ \mathcal{S} q = \gamma_{\Gamma_0^R} \circ \tilde{\Lambda} q$, where $\gamma_{\Gamma_0^R}$ is the trace operator.
        \end{proof}
    }
    
	\section{Numerical Solution Scheme and Reconstruction Method}\label{Sec4}
	{
        In this section, we discuss the discretization of the unbounded locally perturbed variational problem \ref{prob_Var1}, after applying the Bloch-Floquet transform to the variational formulation. To avoid having $\alpha$-dependent spaces $\tilde{H}^1_{\alpha}(\Omega_0^R)$, we will consider functions $w^p_\alpha \in \tilde{H}^1_p(\Omega_0^R)$, where $\tilde{H}^1_p(\Omega_0^R)$ is the space with $\alpha = 0$, instead of $w_\alpha \in \tilde{H}^1_{\alpha}(\Omega^R)$, since they can be identified by $w_\alpha (\tx, x_d) = e^{i \alpha \cdot \tx} w^p_\alpha(\tx, x_d)$. As the gradient $\nabla w_\alpha$ transforms to $(\nabla_x - \ii \alpha) w^p_\alpha$, the $\alpha$-quasi-periodic variational problem \eqref{eq_VarPointwise} for $w_\alpha$  is equivalently reformulated for ${w}^p_\alpha$ as 
        \begin{align}\label{eq_VarPointwiseAlpha}
            a'_\alpha(w^p_\alpha, \overline{v}) 
            &:=
            \int_{\Omega_0^R} \nabla_x w^p_\alpha \cdot \nabla_x \overline{v} 
            - \ii w^p_\alpha  \alpha \cdot \nabla_x  \overline{v} 
            + \ii \alpha \cdot \nabla_x  w^p_\alpha \overline{v} 
            + |\alpha|^2 w^p_\alpha \overline{v} 
            - k^2 n_p^2 w^p_\alpha \overline{v}
            \dd x
            \nonumber
            \\
            &\quad \quad \quad - \int_{\Gamma_0^R} T_\alpha(w^p_\alpha \big|_{\Gamma_0^R}) \overline{v} \big|_{\Gamma_0^R} \dd S
            =  \int_{\Omega_0^R} (\J f)(\alpha, \cdot) \overline{v} \dd x
        \end{align}
        for every $v \in \tilde{H}^1_p(\Omega_0^R)$, where the Dirichlet-to-Neumann operator $T_\alpha$ is defined in the same way, since the Fourier coefficients do not change. 
        We set
        \begin{equation*}
            b'(w, v)
            := - k^2 \int_{I_{}}\int_{\Omega_0^R} e^{\ii \alpha \cdot \tx} q (\J^{-1} w) \overline{v} \dd x \dd \alpha
        \end{equation*}
        for $w$, $v \in L^2(I_{}; \tilde{H}^1(\Omega_0^R))$,
        such that we can write the transformed \eqref{eq_VarUnterIntegral} problem as
        \begin{equation}\label{eq_perProblemUnderIntegral}
            \int_{I_{}} a'_\alpha(w(\alpha, \cdot), v(\alpha, \cdot)) \dd \alpha + b'(w, v)
            = \int_{I_{}}  \int_{\Omega_0^R} \J f \overline{v} \dd x \dd \alpha
            .
        \end{equation}
        Due to the perturbation, the sesquilinear form $b'$ couples the $\alpha$-quasi-periodic components of the transformed solution.
    }
    \subsection{Discretization of the scattering problem}
	{
        In this section, we discretize the variational problem \eqref{eq_perProblemUnderIntegral} as a family of problems, solved by finite elements method. 
        Let $\mathcal{T}$ be the triangulation of $\overline{\Omega_0^R} = [- \pi, \pi]^{d-1} \times [0, R]$, consisting of $2^{d \times M}$ hypercubes that satisfy $\overline{\Omega_0^R} = \bigcup_{T \in \mathcal{T}} \overline{T}$, where $M \in \N$ stands for refinement cycles. 
        Let ${\tilde{M}}$ be the number of nodal points $\{ x^{m} \}_{m=1}^{\tilde{M}} \subset \overline{\Omega^R_0}$, which are equidistant in every direction, and $\{ \phi_{\tilde{M}}^{m} \}_{m=1}^{\tilde{M}}$ the piecewise linear nodal functions, where  $\phi_{\tilde{M}}^{m}$ equals to one at the $m$-th nodal point $x^{m} $, and which equals to zero for other nodal points. Since the solution vanishes on the boundary $\Gamma_0^0$, we do not consider the nodal points there.
        Define the uniformly distributed grid points for $n = 1,2,\ldots,N^{d-1}$ as
        $
            \alpha_N^{n}
            := -\tfrac{1}{2} + \tfrac{1}{2N} + \tfrac{n}{N}
        $
        in the case of $d = 2$ and
        \[
            \alpha_N^{n}
            := \left(-\frac{1}{2} + \frac{1}{2N} + \frac{\lfloor\nicefrac{(n-1)}{N}\rfloor}{N}, -\frac{1}{2} + \frac{1}{2N} + \frac{(n-1) \mod N}{N} \right)
        \]
        in the case of $d=3$ as well as the nodal basis of functions $\{ \psi_N^{n} \}_{n=1}^{N^{d-1}}$, where $\psi_N^{n}$ equals to one on 
        $
            I_N^n
            :=  \alpha_N^{n} + [  \nicefrac{-1}{2N},  \nicefrac{1}{2N}]^{d-1}
        $
        and zero, otherwise. The finite element space $V_{N,{\tilde{M}}}$ is defined as
        \begin{equation}\label{eq_FiniElemRaum}
            V_{N,{\tilde{M}}}
            := 
                \left\{
                    \tilde{v}(\alpha, x)
                    = \sum_{n=1}^{N^{d-1}} \sum_{m=1}^{\tilde{M}} v^{n,m} e^{\ii \alpha_N^{n} \cdot x^{m} } \psi_N^{n}(\alpha) \phi_{\tilde{M}}^m(x)
                    , \ 
                    {v}^{n,m} \in \CC
                \right\}
            ,
        \end{equation}
        and we seek for a finite element solution $\tilde{w}\in V_{N,{\tilde{M}}}$, which solves
        \begin{equation*}
            \int_{I_{}} a'_\alpha(\tilde{w}(\alpha, \cdot), \tilde{v}(\alpha, \cdot)) \dd \alpha + b'(\tilde{w}, \tilde{v})
            = \int_{I_{}}  \int_{\Omega_0^R} \J f  \overline{\tilde{v}}\dd x \dd \alpha
            \quad\text{for all }
            \tilde{v} \in V_{N,{\tilde{M}}}
            .
        \end{equation*}
        For a function $\tilde{w} \in V_{N,{\tilde{M}}}$, the inverse operator $\J^{-1}$ of the Bloch-Floquet transform equals to the trapezoidal rule for integration, since
        \begin{align*}
            \J^{-1} \tilde{w}
            &= \sum_{j=1}^{N^{d-1}} \int_{I^n_N} \sum_{n=1}^{N^{d-1}} \sum_{m=1}^{\tilde{M}} e^{\ii \alpha_N^n \cdot x^{m}} {w}^{n,m} \psi_N^{n}(\alpha) \phi_{\tilde{M}}^m(x)
            \dd \alpha
            \\
            &= \frac{1}{N^{d-1}} \sum_{n=1}^{N^{d-1}} \sum_{m=1}^{\tilde{M}} e^{\ii \alpha_N^n \cdot x^{m}} {w}^{n,m}\phi_{\tilde{M}}^m(x)
            \\
            &=: \frac{1}{N^{d-1}} \sum_{m=1}^{\tilde{M}} u_{\tilde{M}}^{m} \phi_{\tilde{M}}^m(x)
            =: u_{\tilde{M}}(x)
            \\
            &=: \JRand_N^{-1} (\{ {w}^{n,m} \}_{n=1,\ldots,N^{d-1},m=1,\ldots,\tilde{M}})
            .
        \end{align*}
        For $n \in \{1, \ldots, N^{d-1}\}$, we approximate the value $\int_{I^n_N} a'_{\alpha} (\phi_{\tilde{M}}^l , \phi_{\tilde{M}}^m )  \dd \alpha$ by $\tfrac{1}{N^{d-1}} a'_{\alpha_N^n} (\phi_{\tilde{M}}^l , \phi_{\tilde{M}}^m )$ and $b'(\tilde{w}, \psi_N^n\phi_{\tilde{M}}^m)$ by
        \[
            b_N^n(\tilde{w}, \phi_{\tilde{M}}^m)
            :=
             \frac{- k^2}{N^{d-1}} \int_{\Omega_0^R} e^{\ii \alpha_N^n \cdot \tx} q \JRand_N^{-1} (\{ {w}^{n,m} \}) \phi_{\tilde{M}}^m \dd x
            = \frac{1}{N^{d-1}}\sum_{l=1}^{\tilde{M}} u_{\tilde{M}}^{l} b_N^n(\phi_{\tilde{M}}^l , \phi_{\tilde{M}}^m )
            .
        \]
        Thus, the discrete solution 
        \begin{align*}
            W 
            &=\left({w}^{1,1}, \ldots {w}^{1,{\tilde{M}}},{w}^{2,1}, \ldots , {w}^{{N^{d-1}},{\tilde{M}}}, u_{\tilde{M}}^{1}, \ldots, u_{\tilde{M}}^{{\tilde{M}}} \right)
            \\
            &= :\left(W_{1}, \ldots , W_{{N^{d-1}}}, U \right)
            \in \CC^{({N^{d-1}+1) \times {\tilde{M}}}}
            ,
        \end{align*}
        solves the linear system
        \begin{align*}
            \sum_{l=1}^{\tilde{M}} {w}^{n,l} a'_{\alpha_N^n} (\phi_{\tilde{M}}^l , \phi_{\tilde{M}}^m ) 
            + \sum_{l=1}^{\tilde{M}} u_{\tilde{M}}^{l} b_N^n (\phi_{\tilde{M}}^l , \phi_{\tilde{M}}^m ) 
            &
            = F_{m,n}
            &&\quad\text{for }
            m=1,\ldots,{\tilde{M}}
            ,
            \ 
            n=1,\ldots,N^{d-1}
            ,
            \\
            u_{\tilde{M}}^{m} - \frac{1}{N^{d-1}}\sum_{n=1}^{N^{d-1}} e^{- \ii \alpha_N^{n} \cdot x^{m}} {w}^{n,m}
            &=0
            &&\quad\text{for }
            m=1,\ldots,{\tilde{M}}
            ,
        \end{align*}
        where for $n=1,\ldots,N^{d-1}$ and $m=1,\ldots,{\tilde{M}}$ the discrete right hand side is defined by
        \[
            F_{m,n}
            := N^{d-1}\int_{I^n_N}\int_{\Omega_0^R} \J f(\alpha, \cdot)\phi_{\tilde{M}}^m  \dd x \dd \alpha
            .
        \]
        The matrix representation is given by
        \[
            \left( 
                \begin{array}{rr}
                    A       & B        \\
                    C     & I_{\tilde{M}}
                \end{array}
            \right)
            W
            := 
            \left( 
                \begin{array}{rrrrrr}
                    A_1     & 0         & \ldots            & 0         & B_1        \\
                    0       & A_2       & \ldots            & 0         & B_2        \\
                    \vdots  & \vdots    & \ddots            & \vdots    & \vdots     \\
                    0       & 0         & \ldots            & A_{N^{d-1}}       & B_{N^{d-1}}        \\
                    C_1     & C_2       & \ldots            & C_{N^{d-1}}       & I_{\tilde{M}}
                \end{array}
            \right)
            \left( 
                \begin{array}{r}
                    W_1     \\
                    W_2     \\
                    \vdots  \\
                    W_{N^{d-1}}   \\
                    U
                \end{array}
            \right)
            = 
            \left( 
                \begin{array}{r}
                    F_1     \\
                    F_2     \\
                    \vdots  \\
                    F_{N^{d-1}}   \\
                    0
                \end{array}
            \right)
            =: F
            ,
        \]
        where $F_n := \left(F_{1,n}, \ldots , F_{\tilde{M},n} \right)$ and the matrices $A_n$, $B_n$ and $C_n$ are defined as
        \begin{align*}
            A_n(m, l) 
            &= a'_{\alpha_N^n} (\phi_{\tilde{M}}^l , \phi_{\tilde{M}}^m )
            \\
            B_n(m, l) 
            &= b_N^n (\phi_{\tilde{M}}^l , \phi_{\tilde{M}}^m )
            \\
            C_n
            &= \frac{-1}{N^{d-1}} \operatorname{diag} \left(
            	e^{-\ii \alpha_N^{n} \cdot \tx^{1}}
            	, \ldots, 
            	e^{-\ii \alpha_N^{n} \cdot \tx^{{\tilde{M}}}}
            	 \right)
            .
        \end{align*}
        
        The error analysis is out of scope of this paper, we refer the reader to \cite{LechleiterZhang2017} and \cite{LechleiterZhang2017b}. But we note that considering \Cref{thm_regularity_perturbed} we can actually improve the convergence rate of the discrete inverse of the Bloch-Floquet operator, if the right hand side is smooth enough. For that, one has to find a variable transform $g : I_{} \to I_{}$, such that the integrand of
        \[
            \int_{I_{}} w(\alpha, x) \dd \alpha
            = \int_{I_{}} w(g(t), x)\ |\det D g(t)| \dd t
        \]
        is a smooth and periodic function on $I_{}$. It is well-known that the trapezoidal rule is converging very fast in the case of smooth periodic functions. 
        For $d=2$, one can choose a function $g \in C^{\infty}(I_{}; I_{})$, such that $g'(\hat{t}) \geq 0$ and all of the derivatives $\frac{d}{d t^m} g(\hat{t})$, $m \in \N$, vanish at $\hat{t}$, where $g(\hat{t}) = \hat{\alpha}$ holds. In this case, one gets convergence of order $\mathcal{O}(N^{-n} + 2^{-2\times M})$ for some $n \in \N_0$ w.r.t.\ the $L^2(\Omega_0^R)$-norm, if the right hand side is smooth enough (see \cite{Zhang2018} for details).

        The unperturbed system matrix $A$ is a block diagonal matrix consisting of the blocks $A_i$, $i=1,\ldots,N^{d-1}$. This emphasizes to invert the matrix block-wise using {GMRES} and the incomplete LU-decomposition for every block as the preconditioner. Furthermore, this allows the distribution of block-wise inversion tasks over a cluster of computers using \emph{Message Passing Interface} (MPI). To utilize the special structure of A, we first solve the Schur complement for $U$
        \[
            \left[I_M - B A^{-1} C \right] U = -  C A^{-1} F
            ,
        \]
        and in a second step, we solve the equation $A W = F - BU$.
    }
	
	\subsection{Regularization by inexact Newton method}
	{
        In this section, we summarize the regularization scheme for the problem. 
        Since the image spaces of the operators $\Lambda$ and $\mathcal{S}$ are not Hilbert spaces, we adjust these operators first.
        For that, we discretize $L^2(\Omega_0^{R_0})$ by the linear span of $N_f$ nodal functions $\{f_{R, m}, f_{I, m} \}_{m=1}^{N_f}$, which are locally constant with the value of either zero or $1$ in the case of $f_{R, m}$, or, $\ii$ in the case of $f_{I, m}$. All of the functions for the real part $\{f_{R, m}\}_{m=1}^{N_f}$ and all of the functions imaginary part $\{f_{I, m}\}_{m=1}^{N_f}$ are chosen to have disjoint support, such that it holds $\sum_{m=1}^{N_f}f_{R, m} =1$ and $\sum_{m=1}^{N_f}  f_{I, m} =\ii$.
        We define a modified operator $\tilde{\Lambda} : Q \to L^2(\Omega^R_0)^{2 \times N_f}$ of $\Lambda : Q \to \mathcal{L}(L^2(\Omega_0^{R_0}), L^2(\Omega^R_0))$, where a perturbation $q \in Q$ is mapped to the $2 \times N_f$ solutions of the variational problem \ref{prob_Var1} for the corresponding right hand sides. Analogously, we define $\tilde{\mathcal{S}} : Q \to L^2(\Gamma^R_0)^{2 \times N_f}$, which maps the perturbation to the traces of these solutions. Since both operators $\tilde{\Lambda}$ and $\tilde{\mathcal{S}}$ map between Hilbert spaces, we can apply the regularization method CG-REGINN.
        
        Assume now that
        \begin{align}\label{eq_SinverseProblem}
            \tilde{\Lambda} q^+
            &= U^+
            &&\in \tilde{H}^1(\Omega_0^R)^{2 \times N_f}
            &\quad\text{for }
            &q^+ \in Q \subseteq L^2(\Omega_0^{R_0})
            ,
            \\
            \nonumber
            \tilde{\mathcal{S}} q^+
            &= U^+\big|_{\Gamma_0^R}
            &&\in  H^{\nicefrac{1}{2}}(\Gamma_0^R)^{2 \times N_f}
            &\quad\text{for }
            &q^+ \in Q \subseteq L^2(\Omega_0^{R_0})
            .
        \end{align}
        We briefly summarize the regularization scheme {CG-REGINN} (``REGularization based on INexact Newton iteration'') stated and analyzed by Rieder in \cite{Rieder2005},
        which we propose for the inversion. We will only consider the first inverse problem in \eqref{eq_SinverseProblem} for the summary.
   
        We have given the noisy version $U^{\varepsilon}$ of the exact measurement $U^+$ with the relative noise level $\varepsilon \in (0, 1)$, i.e., $||U^{\varepsilon} - U^+|| \leq \varepsilon ||U^+|| \approx \varepsilon ||U^{\varepsilon}||$, which we assume to know a-priori.
        The algorithm generates a sequence $\{q_m \}_{m \in \N_0} \subseteq Q$ of approximations of $q^+$, starting with the initial guess $q_0 \in Q$. If we write $q^+ = q_m + s_m^+$ for each $m \in \N_0$, the best update $s_m^+$ solves the linearized problem
        \[
            \tilde{\Lambda}'(q_m)[s_m^+]
            = U^+ - \tilde{\Lambda}(q_m) - E(q^+, q_m)
            =: b_m^+
            ,
        \]
        where $E(q^+, q_m)$ is the linearization error. Since we do not know the linearization error, we only know the perturbed right hand side $b_m^\varepsilon := U^{\varepsilon} - \tilde{\Lambda}(q_m)$ with the upper bound $||b_m^\varepsilon - b_m^+|| \leq \varepsilon ||U^{\varepsilon}|| + \mathcal{O}(||q^+ - q_m||)$  for the noise level.
        
        {CG-REGINN} applies the regularization method of conjugate gradients (CG) for the linearized problem and stops, when the relative linear residuum is smaller than a tolerance times the non-linear residuum. CG creates an inner iteration that computes a sequence of approximations $\{s_{m,i} \}_{i \in \N_0}$ of $s^+_m$. The inner loop is terminated, when $||\tilde{\Lambda}'(q_m)[s_{m,i}] - b_m^\varepsilon|| < \mu_m ||b_m^\varepsilon||$ for a tolerance $\mu_m \in (0,1)$ is satisfied for the first time, which index we call $i_m$. Then, we use backtracking, to get $\tilde{s}_{m,i_m} := \beta \tilde{s}_{m,i_m} + (1-\beta) \tilde{s}_{m,i_m-1}$, where $\beta \in [0,1]$ is chosen, such that $||\tilde{S}'(q_m)[\tilde{s}_{m,i}] - b_m^\varepsilon|| = \mu_m ||b_m^\varepsilon||$. We define the update as $s_{m} :=\tilde{s}_{m,i_m}$ and set $q_{m+1} = q_m + s_{m}$, until the discrepancy principle with $\tau > 0$ is satisfied for the outer loop.
        
        Considering the suggestion in \cite{Rieder2005}, we chose $\mu_1 = \mu_2 = \mu_{start} = 0.55$, $\gamma = 0.9$, $\mu_{max}=0.99$, and
        \[
            \mu_m
            = \mu_{max} \max\left\{\tau \varepsilon ||U^{\varepsilon}|| / ||U^{\varepsilon} - \tilde{\Lambda}(q_m)||
            ,\ 
            \tilde{\mu}_m
            \right\}
            ,
        \]
        where
        \[
            \tilde{\mu}_m
            := \left\{
                \begin{array}{ll}
                    1 - \frac{i_{m-2}}{i_{m-1}} (1 - \mu_{m-1}) 
                    ,
                    & \text{for } i_{m-1} > i_{m-2}
                    \\
                    \gamma \mu_{m-1}
                    ,                    
                    & \text{else}
                    .
                    \\
                \end{array}
            \right.
        \]
        Taking the results in \cite{Hansen2018} into account, we use the adjoint matrix of the discretized problem for the inner loop of the numerical reconstruction, instead of the discretization of the theoretical adjoint of the Fr\'echet derivative, to have a more stable inversion.
    }	
	\section{Numerical Examples}\label{Sec5}
    {
        In this section, we present some numerical results for the Bloch transform based method and the two inverse problems.
        We note at this point that the \textrm{deal.II} library does not support complex numbers, such that the values of the functions are considered as elements of $\R^2$. In this case, we get a system of two partial differential equations with some couplings, which double the number of degrees of freedom.
        \subsection{Example for the Bloch transform based method}
            For the first and second example, we choose $d=2$, $R=5$, $k=\sqrt{0.4}$, the cut-off of the Fourier expansion of the boundary for $|j| \leq 300$, and
            \[
                u_1 (x_1, x_2)
                = \exp({-\tfrac{1}{10}(x_1 - 1)^2 + \tfrac{1}{10}(x_2 - 5)^2}) \frac{x_2}{5}
            \]
            as well as
            \[
                u_2 (x_1, x_2)
                = \frac{\ii}{4} \left(H^{(1)}_0 (k(x_1^2 + (x_2+7)^2) ) - H^{(1)}_0  (k(x_1^2 + (x_2+9)^2) ) \right) \frac{x_2}{5}
            \]
            as the reference solutions. The Bloch-Floquet transformed function of the second solution $u_2$ can be approximated by
            \begin{align*}
                \JRand_{\R} (u_2)(\alpha, x_1, x_2)
                &=  \sum_{j \in \Z}  e^{- \ii m x_1 + \ii \sqrt{k^2 - (\alpha + m)^2}(x_2+8)} \operatorname{sinc}(\sqrt{k^2 - (\alpha + m)^2}) 
                \\
                &\approx \sum_{m = -100}^{100} e^{- \ii m x_1 + \ii \sqrt{k^2 - (\alpha + m)^2}(x_2+8)} \operatorname{sinc}(\sqrt{k^2 - (\alpha + m)^2})
                ,
            \end{align*}
            where $\operatorname{sinc}(t) := \nicefrac{\operatorname{sin}(t)}{t}$ is a smooth function. The transformed function of $u_1$, we simply approximate by
            \begin{align*}
                \JRand_{\R} (u_1)(\alpha, x_1, x_2)
                &= \sum_{j \in \Z} \exp({-\tfrac{1}{10}(x_1 - 1 + 2 \pi j)^2 + \tfrac{1}{10}(x_2 - 5)^2}) \frac{x_2}{5} e^{2 \pi \ii \alpha (j + x)}
                \\
                &\approx \sum_{j = -30}^{30} \exp({-\tfrac{1}{10}(x_1 - 1 + 2 \pi j)^2 + \tfrac{1}{10}(x_2 - 5)^2}) \frac{x_2}{5} e^{2 \pi \ii \alpha (j + x)}
                ,
            \end{align*}
            since this function is decaying fast.
            Because of the extra factor $\nicefrac{x_2}{5}$, both do not satisfy the Neumann boundary condition, such that we add some correction factors $r_i$, $i=1,2$,
            \[
                r_i
                := 
                \frac{\partial }{\partial x_2} u_i - T(u_i)
                \text{ on }
                \R \times \{5\}
                \quad\text{for }
                i=1,2
                ,
            \]
            where $r_2$ can be simplified to
            \begin{equation*}
                r_2(x_1, 5) 
                = \frac{i}{20} \left(H^{(1)}_0 (k(x_1^2 + 144) ) - H^{(1)}_0  (k(x_1^2 + 196) ) \right)
                .
            \end{equation*}
            For the unperturbed refractive index, we take the function
            \[
                k^2 n^2_p
                =
                    \left\{
                        \begin{array}{ll}
                            0.8, & x \in ([\nicefrac{-3}{2}, \nicefrac{3}{2}] \times [0, \nicefrac{9}{2}] \cup [- \pi, \pi] \times [0, \nicefrac{7}{2}]) \setminus [-1, 1] \times [1, 3]\\
                            0.8 + 0.4\ii, & x \in [-1, 1] \times [1, 3] \\
                            1, & \, \textrm{else,} \\
                        \end{array}
                    \right.
            \]
            and the perturbation $k^2q$ is given by
            \[
                k^2 q
                =
                    \left\{
                        \begin{array}{ll}
                            2.2, & x \in [\nicefrac{-1}{2}, 1] \times [1, \nicefrac{7}{2}] \cup [- 2, 1] \times [1, 2]\\
                            0, & \, \textrm{else.} \\
                        \end{array}
                    \right.
            \]
            In \Cref{image_nSquareAndQ} both parameter are visualized. 
            We set $2^{2 \times M}$ as the number of cubic cells the domain is discretized in and $N$ the number of points for the discretization of the interval $I_{}$. Note that $2^{2 \times M}=\SI{65536}{}$ corresponds to $\SI{132098}{}$ degrees of freedom, since we have a system of two partial differential equations. The relative tolerance for GMRES is chosen to be $10^{-10}$. In \Cref{table1} and \Cref{table2}  one can see the relative $L^2(\Omega^R_0)$-errors for the two examples and in \Cref{table12} the computation time for Example 1 using three computers  ({Intel i7-4790}, $8 \times 3.6$GHz cores, $32$GB memory) in parallel. In both cases we use the variable transformation $g : I_{} \to I_{}$, which is defined as
            \[
                g(t) :=
                    \left\{
                        \begin{array}{ll}
                            \nicefrac{\phi_{\nicefrac{-1}{2}, \hat{k}}(t)}{\phi_{\nicefrac{-1}{2}, \hat{k}}(\hat{k})}, & t \in [\nicefrac{-1}{2}, -\hat{k}]
                            \\
                            \nicefrac{\psi_{-\hat{k}, \hat{k}}(t)}{\psi_{-\hat{k}, \hat{k}}(\hat{k})}, & t \in [-\hat{k}, \hat{k}]
                            \\
                            \nicefrac{\phi_{\nicefrac{1}{2}, \hat{k}}(t)}{\phi_{\nicefrac{1}{2}, \hat{k}}(\hat{k})}, & t \in [\hat{k}, \nicefrac{1}{2}]
                        \end{array}                        
                    \right.
                ,
            \]
            where $\hat{k} := |k - \lfloor k + 0.5\rfloor |$,
            \begin{equation*}
                \phi_{l, m}(t)
                := \int_l^t \exp\left({\frac{-(s - l)^2}{9 (s - m)^2} }\right) \dd s
                \text{ and }
                \psi_{l, m}(t)
                := \int_l^t \exp\left({\frac{-1}{9 (s - l)^2(s - m)^2} }\right) \dd s
                .
            \end{equation*}
            If we use the identity as variable transform in the case of Example 1, then the error would decrease faster for smaller $N$, since the function $\J u_1(\cdot, x)$ is $C_p^\infty(I_{})$ for every $x \in \Omega^R_0$. For $N=8$ we would already see near as good error values as for $N=256$ in \Cref{table1}. But in the case of the second example, the variable transform lets the error decrease much faster w.r.t.\ $\alpha$, since the second part of the decomposition  of $\J u_2$ shown in \Cref{thm_regularity_perturbed} does not vanish and the function has the square-root-line behavior.

            \begin{figure}[!ht]
                \begin{minipage}{\textwidth}
                    \hfill
                    \subfigure[Refractive index $\Re  k^2 n^2_p$.]{
                        \includegraphics[width=0.47\textwidth]{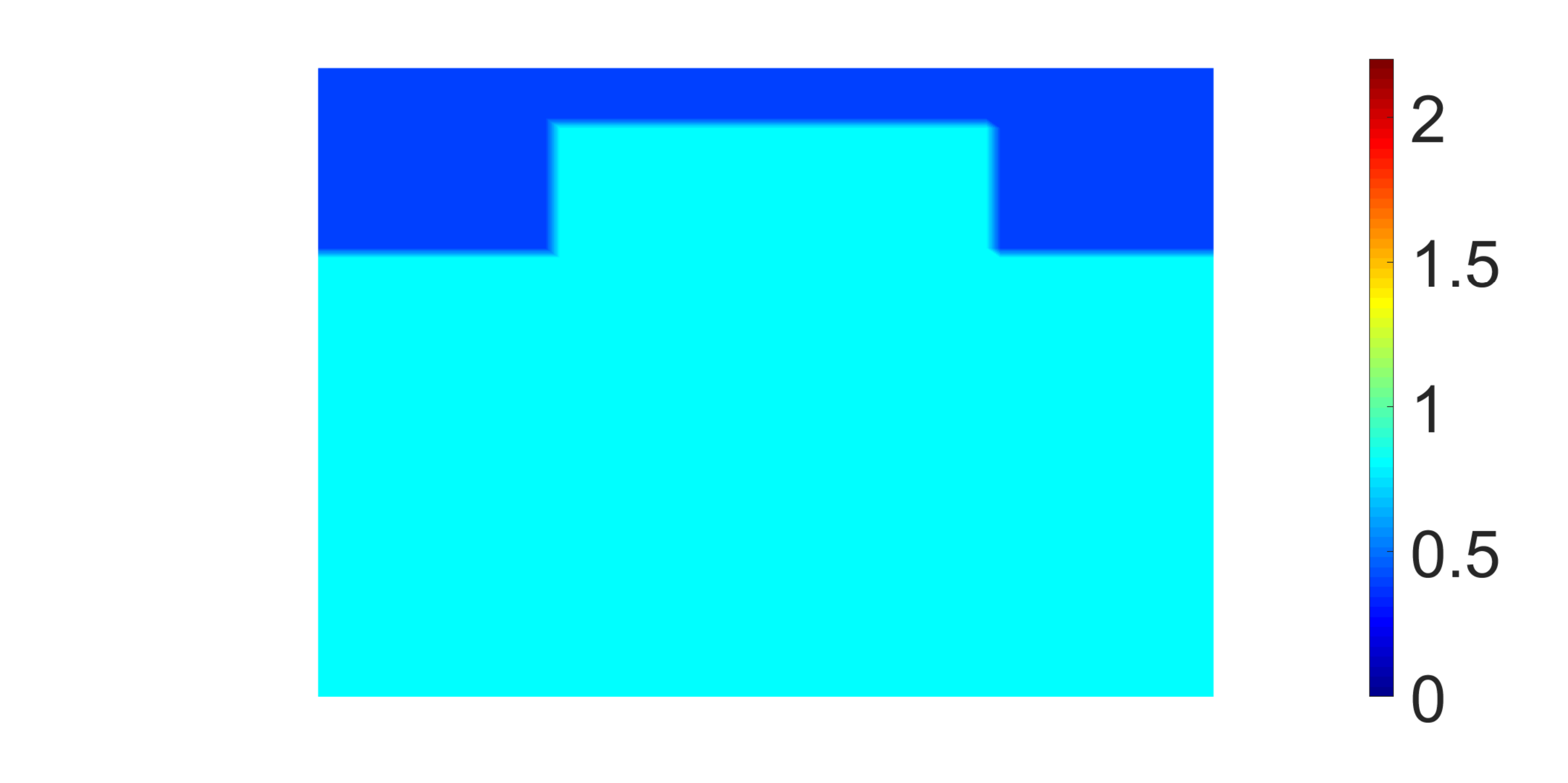}
                    }
                    \hfill
                    \subfigure[Refractive index $\Im  k^2 n^2_p$.]{
                        \includegraphics[width=0.47\textwidth]{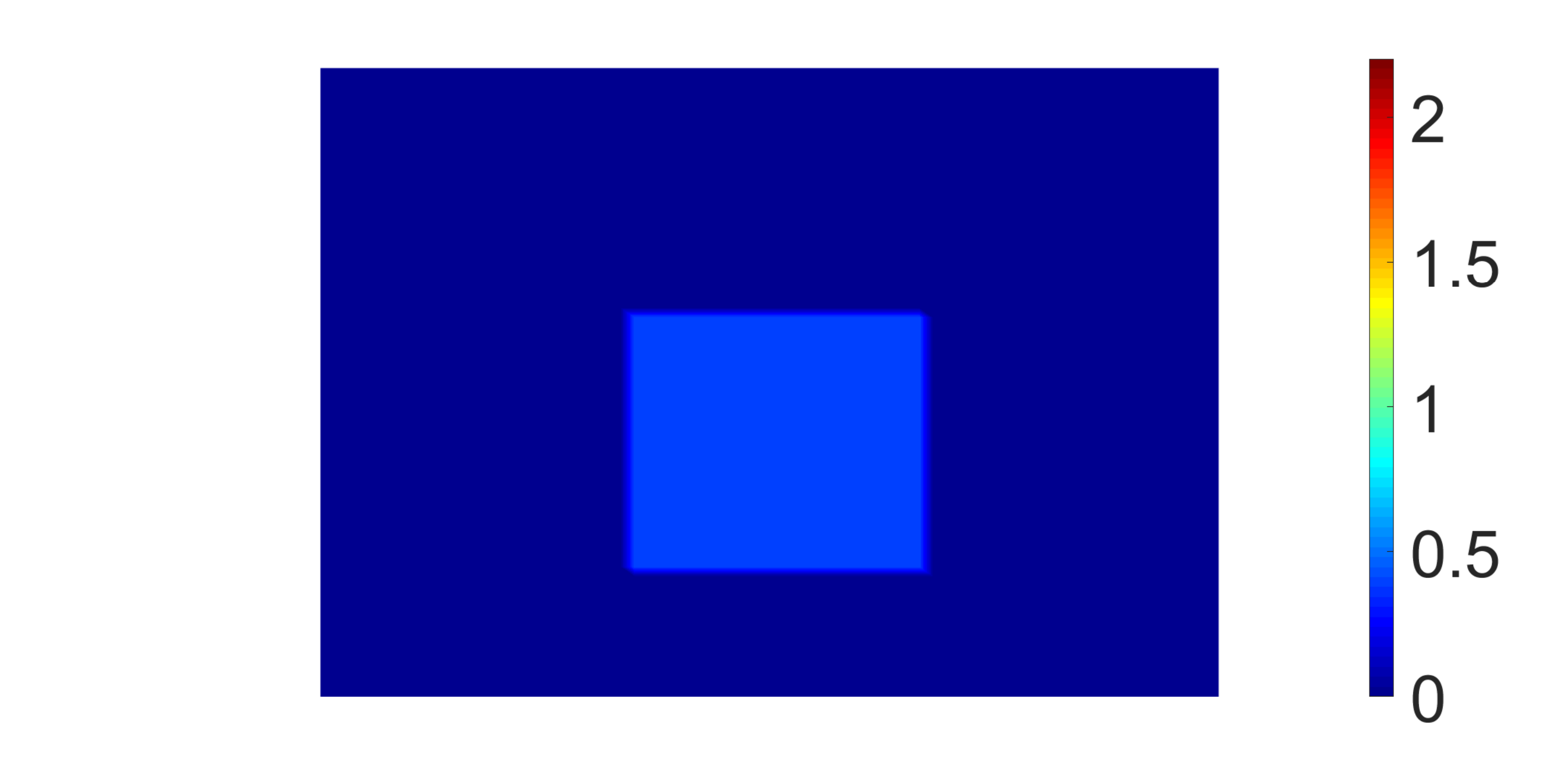}
                    }
                    \hfill
               \end{minipage}\\[1em]
               \begin{minipage}{\textwidth}
                    \hfill
                    \subfigure[Perturbation $\Re  k^2 q$.]{
                        \includegraphics[width=0.47\textwidth]{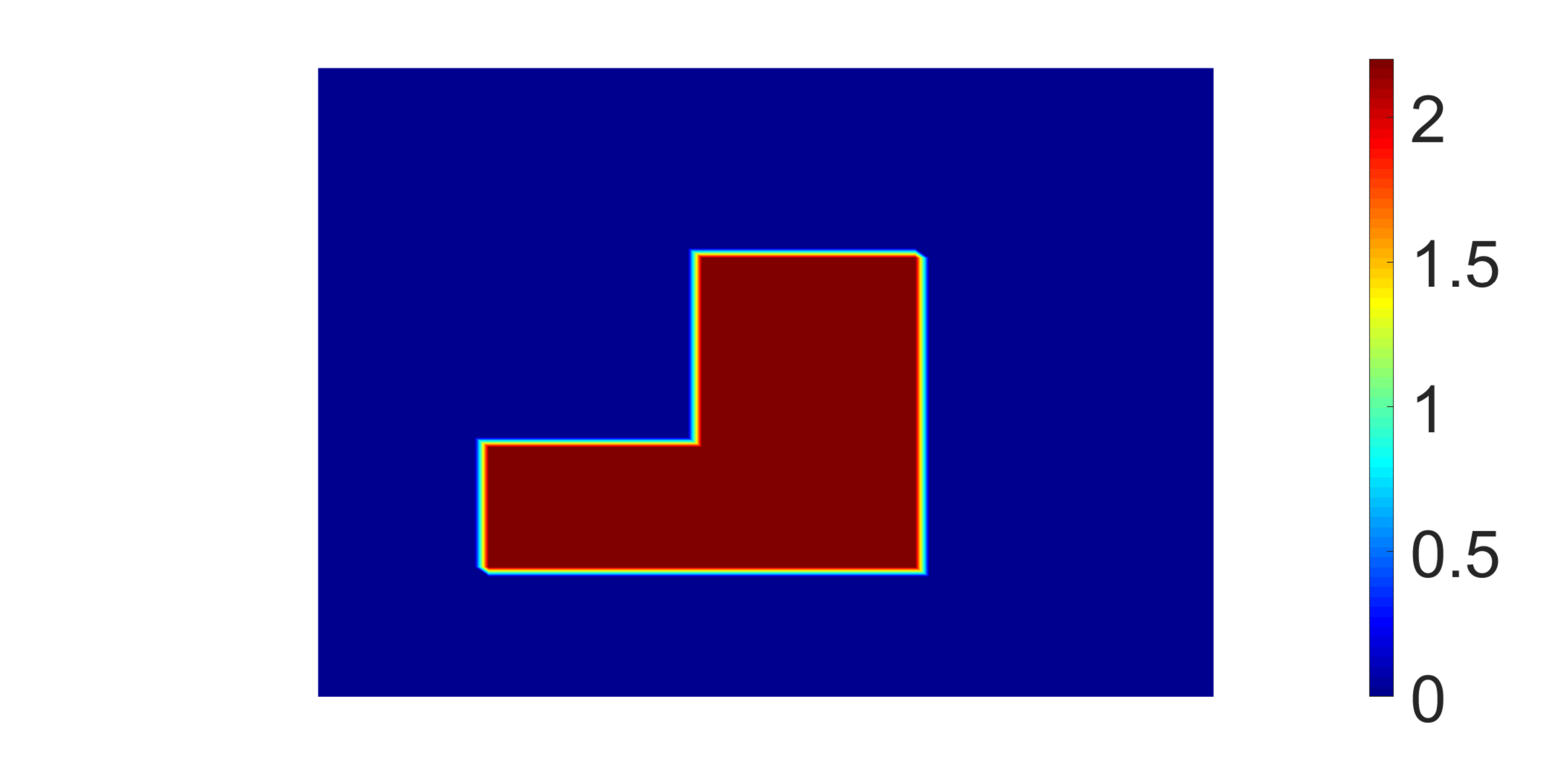}
                    }
                    \hfill
                    \subfigure[Perturbation $\Im  k^2 q$.]{
                        \includegraphics[width=0.47\textwidth]{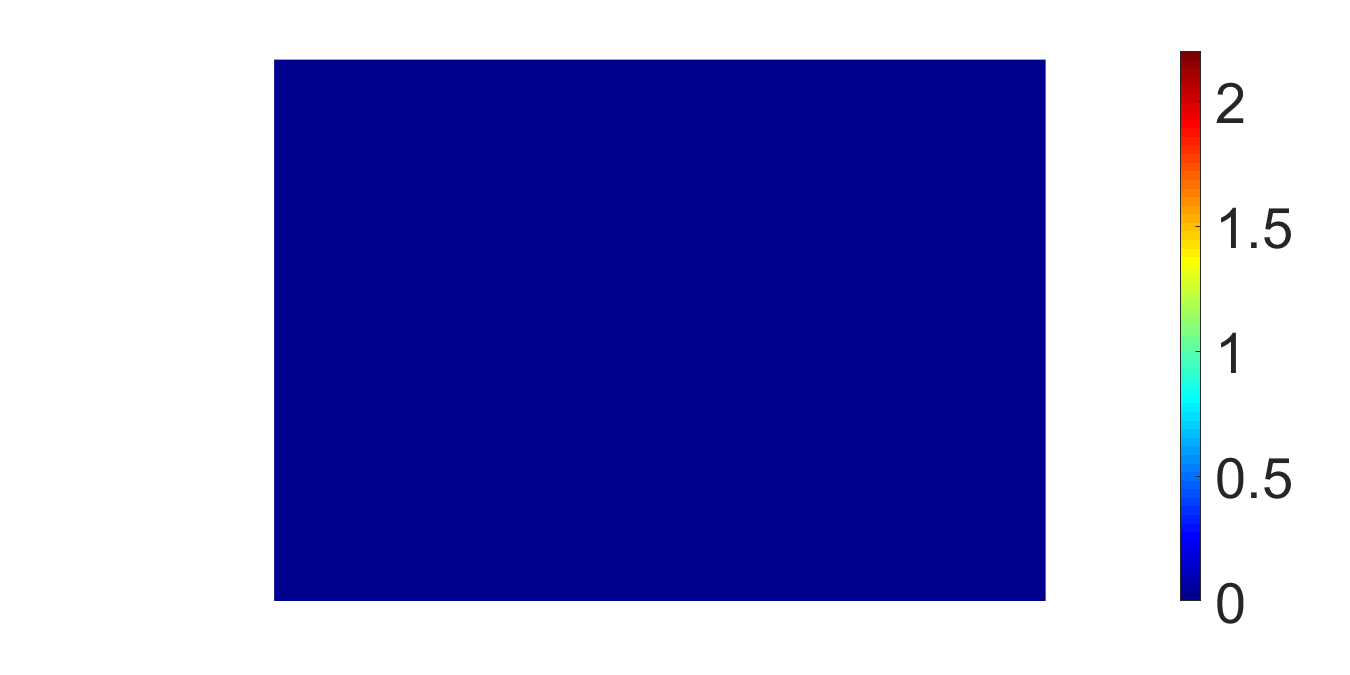}
                    }
                    \hfill
                    \caption{The refractive index and the perturbation for all 2D examples.}
                    \label{image_nSquareAndQ}
                \end{minipage}
            \end{figure}
            
            \begin{table}[H]
                \begin{center}
                    \begin{tabular}{|lr|c|c|c|c|c|c|c|} 
                        \hline
                                             & & $N=\textbf{8}$         & $ N=\textbf{16}$        & $ N=\textbf{32}$        & $ N=\textbf{64}$        & $ N=\textbf{128}$       & $ N=\textbf{256}$
                        \\
                        \hline
                           $2^{2 \times M}=$ & ${\textbf{256}}{}$                     & {6.210e-02}{}    & {1.741e-02}{}    & {1.964e-02}{}    & {1.825e-02}{}    & {1.826e-02}{}    & {1.826e-02}{}
                        \\
                          $2^{2 \times M}=$ & ${\textbf{1\,024}}{}$                     & {5.862e-02}{}    & {8.930e-03}{}    & {6.546e-03}{}    & {4.473e-03}{}    & {4.520e-03}{}    & {4.519e-03}{}
                        \\
                          $2^{2 \times M}=$ & ${\textbf{4\,096}}{}$                     & {5.901e-02}{}    & {9.473e-03}{}    & {3.871e-03}{}    & {1.088e-03}{}    & {1.128e-03}{}    & {1.127e-03}{}
                        \\
                         $2^{2 \times M}=$ & ${\textbf{16\,384}}{}$                     & {5.901e-02}{}    & {9.843e-03}{}    & {3.413e-03}{}    & {2.788e-04}{}    & {2.786e-04}{}    & {2.776e-04}{}
                        \\
                         $2^{2 \times M}=$ & ${\textbf{65\,536}}{}$                     & {5.894e-02}{}    & {9.910e-03}{}    & {3.309e-03}{}    & {1.662e-04}{}    & {6.866e-05}{}    & {6.847e-05}{}
                        \\
                        \hline
                    \end{tabular}
                    \captionof{table}{Relative $L^2(\Omega^R_0)$-error for Example 1.}
                    \label{table1}
                \end{center}
            \end{table}

            \begin{table}[H]
                \begin{center}
                    \begin{tabular}{|lr|r|r|r|r|r|r|r|} 
                        \hline
                                              && $N=\textbf{8}$         & $ N=\textbf{16}$        & $ N=\textbf{32}$        & $ N=\textbf{64}$        & $ N=\textbf{128}$       & $ N=\textbf{256}$
                        \\
                        \hline
                           $2^{2 \times M}=$ & ${\textbf{256}}{}$                       & $\SI{1}{}$         & $\SI{2}{}$         & $\SI{3}{}$         & \SI{6}{}         & $\SI{12}{}$        & $\SI{26}{}$
                        \\ 
                          $2^{2 \times M}=$ & ${\textbf{1\,024}}{}$                       & $\SI{1}{}$         & $\SI{2}{}$         & $\SI{4}{}$         & $\SI{7}{}$         & $\SI{14}{}$        & $\SI{28}{}$
                        \\ 
                          $2^{2 \times M}=$ & ${\textbf{4\,096}}{}$                       & $\SI{4}{}$         & $\SI{7}{}$         & $\SI{13}{}$        & $\SI{25}{}$        & $\SI{47}{}$        & $\SI{93}{}$
                        \\
                         $2^{2 \times M}=$ & ${\textbf{16\,384}}{}$                       & $\SI{36}{}$        & $\SI{70}{}$        & $\SI{139}{}$       & $\SI{265}{}$       & $\SI{539}{}$       & $\SI{1026}{}$
                        \\
                         $2^{2 \times M}=$ & ${\textbf{65\,536}}{}$                       & $\SI{643}{}$       & $\SI{1786}{}$      & $\SI{3337}{}$      & $\SI{6737}{}$      & $\SI{13724}{}$     & $\SI{27689}{}$
                        \\
                        \hline
                    \end{tabular}
                    \captionof{table}{Time in seconds for Example 1.}
                    \label{table12}
                \end{center}
            \end{table}

            \begin{table}[H]
                \begin{center}
                    \begin{tabular}{|lr|c|c|c|c|c|c|c|} 
                        \hline
                                              && $N=\textbf{8}$         & $ N=\textbf{16}$        & $ N=\textbf{32}$        & $ N=\textbf{64}$        & $ N=\textbf{128}$       & $ N=\textbf{256}$
                        \\
                        \hline
                           $ 2^{2 \times M}=$ & ${\textbf{256}}{}$                     & {3.558e-01}{}    & {5.420e-02}{}    & {2.157e-02}{}    & {8.410e-03}{}    & {8.729e-03}{}    & {8.722e-03}{}
                        \\
                          $ 2^{2 \times M}=$ & ${\textbf{1\,024}}{}$                     & {3.651e-01}{}    & {5.247e-02}{}    & {1.429e-02}{}    & {1.900e-03}{}    & {2.200e-03}{}    & {2.195e-03}{}
                        \\
                          $ 2^{2 \times M}=$ & ${\textbf{4\,096}}{}$                     & {3.631e-01}{}    & {5.221e-02}{}    & {1.257e-02}{}    & {4.486e-04}{}    & {5.352e-04}{}    & {5.305e-04}{}
                        \\
                         $ 2^{2 \times M}=$ & ${\textbf{16\,384}}{}$                     & {3.630e-01}{}    & {5.213e-02}{}    & {1.213e-02}{}    & {4.461e-04}{}    & {1.257e-04}{}    & {1.236e-04}{}
                        \\
                         $ 2^{2 \times M}=$ & ${\textbf{65\,536}}{}$                     & {3.644e-01}{}    & {5.195e-02}{}    & {1.203e-02}{}    & {5.003e-04}{}    & {4.726e-05}{}    & {5.279e-05}{}
                        \\
                        \hline
                    \end{tabular}
                    \captionof{table}{Relative $L^2(\Omega^R_0)$-error for Example 2.}
                    \label{table2}
                \end{center}
            \end{table}

        To show some three dimensional examples, we choose $d=3$, $R=5$, $k=\sqrt{0.4}$ and the Fourier expansion cut-off $|j| \leq 10$. For the reference solutions, we choose
        \[
            u_3 (x_1, x_2, x_3)
            = \exp({-\tfrac{1}{10}(x_1 - 1)^2-\tfrac{1}{10}(x_2 - 2)^2 + \tfrac{1}{10}(x_3 - 5)^2}) \frac{x_3}{5}
        \]
        and
        \[
            u_4 (x_1, x_2, x_3)
            = \left[\frac{e^{\ii k (x_1^2 + (x_2+7)^2)}}{4 \pi (x_1^2 + x_2^2 + (x_3+7)^2)}  - \frac{e^{\ii k (x_1^2 + x_2^2 + (x_3+9)^2)}}{4 \pi (x_1^2 + x_2^2 + (x_3+9)^2)} \right] \frac{x_3}{5}
            .
        \]
        We approximate the Bloch-Floquet transformed functions $\JRand_{\R^2} u_3$ and $\JRand_{\R^2} u_4$ by
        \[
            \JRand_{\R^2} (u_4)(\alpha, x_1, x_2)
            \approx \sum_{m \in \Z^2,\ |m|\leq 10} e^{- \ii m \cdot  \tx + \ii \sqrt{k^2 - (\alpha + m)^2}(x_3+8)} \operatorname{sinc}(\sqrt{k^2 - |\alpha + m|^2})
            ,
        \]
        or,
        \begin{align*}
            &\JRand_{\R^2} (u_3)(\alpha, x_1, x_2)
            \\
            &\approx \sum_{j  \in \Z^2,\ |j|\leq 10} \exp({-\tfrac{1}{10}(x_1 - 1 + 2 \pi j_1)^2-\tfrac{1}{10}(x_2 - 2 + 2 \pi j_2)^2 + \tfrac{1}{10}(x_3 - 5)^2}) \frac{x_3}{5} e^{2 \pi \ii \alpha \cdot (j + \tx)}
            ,
        \end{align*}
        respectively, and we add a correction term for the Neumann boundary condition.
            For the unperturbed refractive index, we choose the function
            \[
                k^2 n^2_p
                =
                    \left\{
                        \begin{array}{ll}
                            0.8, & x \in ([\nicefrac{-3}{2}, \nicefrac{3}{2}] \times [1, \pi] \times [0, \nicefrac{9}{2}] \cup [- \pi, \pi]^2 \times [0, \nicefrac{7}{2}]) \setminus [-1, 1]^2 \times [1, 3]\\
                            0.8 + 0.4\ii, & x \in [-1, 1]^2 \times [1, 3] \\
                            1, & \, \textrm{else,} \\
                        \end{array}
                    \right.
            \]
            and the perturbation $k^2q$ is given by
            \[
                k^2 q
                =
                    \left\{
                        \begin{array}{ll}
                            2.2, & x \in [\nicefrac{-1}{2}, 1] \times [0, 1] \times [1, \nicefrac{7}{2}] \cup [- 2, 1] \times [0, 1] \times [1, 2] \cup [\nicefrac{-1}{2}, 1] \times [\nicefrac{-5}{2}, 1] \times [1, 2]
                            \\
                            0, & \, \textrm{else.} \\
                        \end{array}
                    \right.
            \]
            Both parameter are visualized in \Cref{image_reconstruction3D}. The relative tolerance of GMRES is still chosen as $10^{-10}$, and we took the identity for the variable transform $g$ in both cases.
            \begin{table}[H]
                \begin{center}
                    \begin{tabular}{|lr|c|c|c|c|c|c|c|} 
                        \hline
                                              && $N^2=\textbf{16}$         & $ N^2=\textbf{64}$        & $ N^2=\textbf{256}$
                        \\
                        \hline
                           $4^{3 \times M}=$ & ${\textbf{8}}{}$                     & {7.534e-01}{}    & {7.591e-01}{}    & {7.590e-01}{}
                        \\
                          $4^{3 \times M}=$ & ${\textbf{64}}{}$                     & {4.807e-01}{}    & {5.256e-01}{}    & {5.260e-01}{}
                        \\
                          $4^{3 \times M}=$ & ${\textbf{512}}{}$                    & {1.216e-01}{}    & {8.489e-02}{}    & {1.498e-01}{}
                        \\
                         $4^{3 \times M}=$ & ${\textbf{4\,096}}{}$                  & {2.005e-02}{}    & {4.158e-02}{}    & {2.779e-02}{}
                        \\
                         $4^{3 \times M}=$ & ${\textbf{32\,768}}{}$                 & {6.150e-03}{}    & {9.793e-03}{}    & {6.611e-03}{}
                        \\
                         $4^{3 \times M}=$ & ${\textbf{262\,144}}{}$                & {1.577e-03}{}    & {2.450e-03}{}    & -
                        \\
                        \hline
                    \end{tabular}
                    \captionof{table}{Relative $L^2(\Omega^R_0)$-error for Example 3.}
                    \label{table3}
                \end{center}
            \end{table}
            \begin{table}[H]
                \begin{center}
                    \begin{tabular}{|lr|c|c|c|c|c|c|c|c|} 
                        \hline
                                              && $N^2=\textbf{16}$         & $ N^2=\textbf{64}$        & $ N^2=\textbf{256}$        & $ N^2=\textbf{1\,024}$        & $ N^2=\textbf{4\,096}$
                        \\
                        \hline
                           $4^{3 \times M}=$ & ${\textbf{8}}{}$                     & {3.656e-01}{}    & {3.775e-01}{}    & {3.775e-01}{}    & {3.502e-01}{}    & {3.490e-01}{}
                        \\
                          $4^{3 \times M}=$ & ${\textbf{64}}{}$                     & {4.769e-01}{}    & {1.017e-00}{}    & {5.350e-01}{}    & {4.858e-01}{}    & {4.718e-01}{}
                        \\
                          $4^{3 \times M}=$ & ${\textbf{512}}{}$                    & {6.182e-01}{}    & {1.480e-01}{}    & {4.499e-02}{}    & {5.636e-02}{}    & {4.387e-02}{}
                        \\
                         $4^{3 \times M}=$ & ${\textbf{4\,096}}{}$                  & {5.965e-01}{}    & {1.018e-01}{}    & {1.708e-02}{}    & {1.926e-02}{}    & {8.951e-03}{}
                        \\
                         $4^{3 \times M}=$ & ${\textbf{32\,768}}{}$                 & {5.959e-01}{}    & {8.270e-02}{}    & {1.716e-02}{}    & {1.330e-02}{}    & {-}{}
                        \\
                        \hline
                    \end{tabular}
                    \captionof{table}{Relative $L^2(\Omega^R_0)$-error for Example 4.}
                    \label{table4}
                \end{center}
            \end{table}
        
        \subsection{Examples for the reconstruction of the perturbation}
            In this subsection, we give the results, if we reconstruct the perturbation of the periodic refractive index, both shown in \Cref{image_nSquareAndQ}.
            To generate the data, we use the algorithm from above, and refine some of the parameter, such that, we have $\SI{16384}{}$ cells for $\Omega^R_0$, $128$ subintervals of $I_{}$ and with the cut-off of the Fourier expansion of the boundary is $|j| \leq 600$. After that, we interpolate the solution down to $\SI{4096}{}$ cells for $\Omega^R_0$, put some unified distributed noise of $5\%$ on it, and use this as the given data.
            For the reconstruction, we use $\SI{4096}{}$  cells for $\Omega^R_0$, $N=64$, and a cut-off of $300$.
            For the right hand sides, we choose $R_0 = 4.5$, and split the domain $\Omega_0^{R_0}$ into $16$ equal parts. We approximate the $L^2(\Omega_0^{R_0})$ space with $32$ local constant functions $\{f_{R, m}, f_{I, m} \}_{m=1}^{16}$, which are locally constant on the every part with the value of either zero or $1$ in the case of $f_{R, m}$, or,  $\ii$ in the other case, and such that it holds $\sum_{m=1}^{16}f_{R, m} =1$, $\sum_{m=1}^{16}  f_{I, m} =\ii$.
            
            We approximate the perturbation $k^2q$ as a function in the finite element space, which is spanned by the finite elements $\{\phi_{\tilde{M}}^m\}_{m=1}^{\tilde{M}}$ in \eqref{eq_FiniElemRaum}, and stop the outer iteration of REGINN by the discrepancy principle, when the relative discrepancy is smaller than $1.2 \times 0.05$. In \Cref{image_reconstruction} one can see the result of the reconstruction, where relative $L^2(\Omega_0^{^R})$ reconstruction error is about 
            $38.0087\%$
            in the case of $\tilde{\Lambda}$, and  a reconstruction error of 
            $57.1368\%$ 
            in the case of $\tilde{\mathcal{S}}$. The results for the inversion of $\tilde{\Lambda}$ are much better, since it has more data given to work with.
            Furthermore, the quality of the reconstruction  depends highly on the size and the value of the absorption area. 
            The bigger the set $\{\Im n^2_p > 0\}$ and the value inside is, the more the error of the reconstruction decreases.
            \begin{figure}[!ht]
                \begin{minipage}{\textwidth}
                    \hfill
                    \subfigure[Reconstruction of  $\Re  k^2 q$ in the case of $\tilde{\Lambda}$.]{
                        \includegraphics[width=0.47\textwidth]{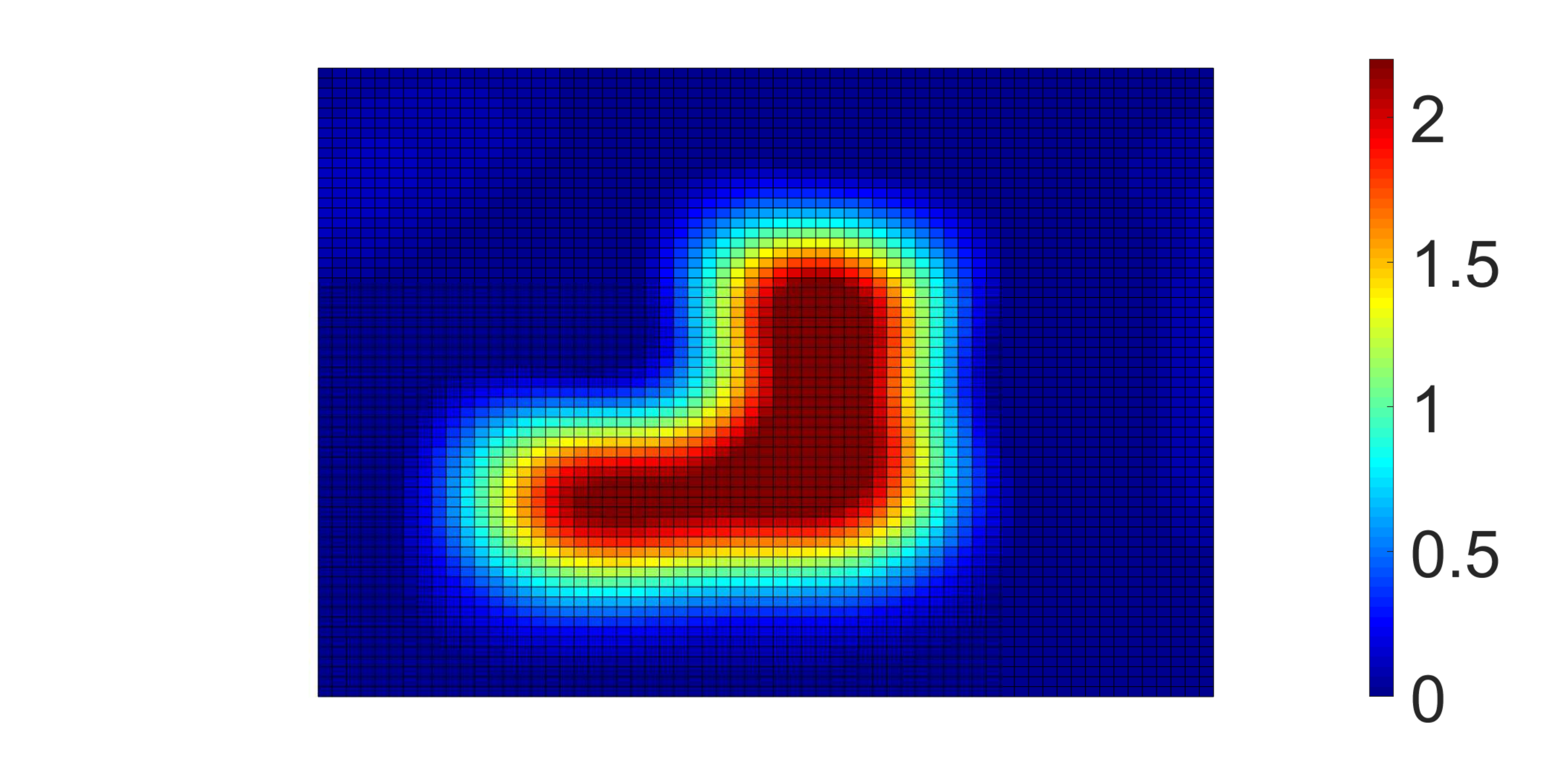}
                    }
                    \hfill
                    \subfigure[Reconstruction of  $\Re k^2 q$ in the case of $\tilde{\mathcal{S}}$.]{
                        \includegraphics[width=0.47\textwidth]{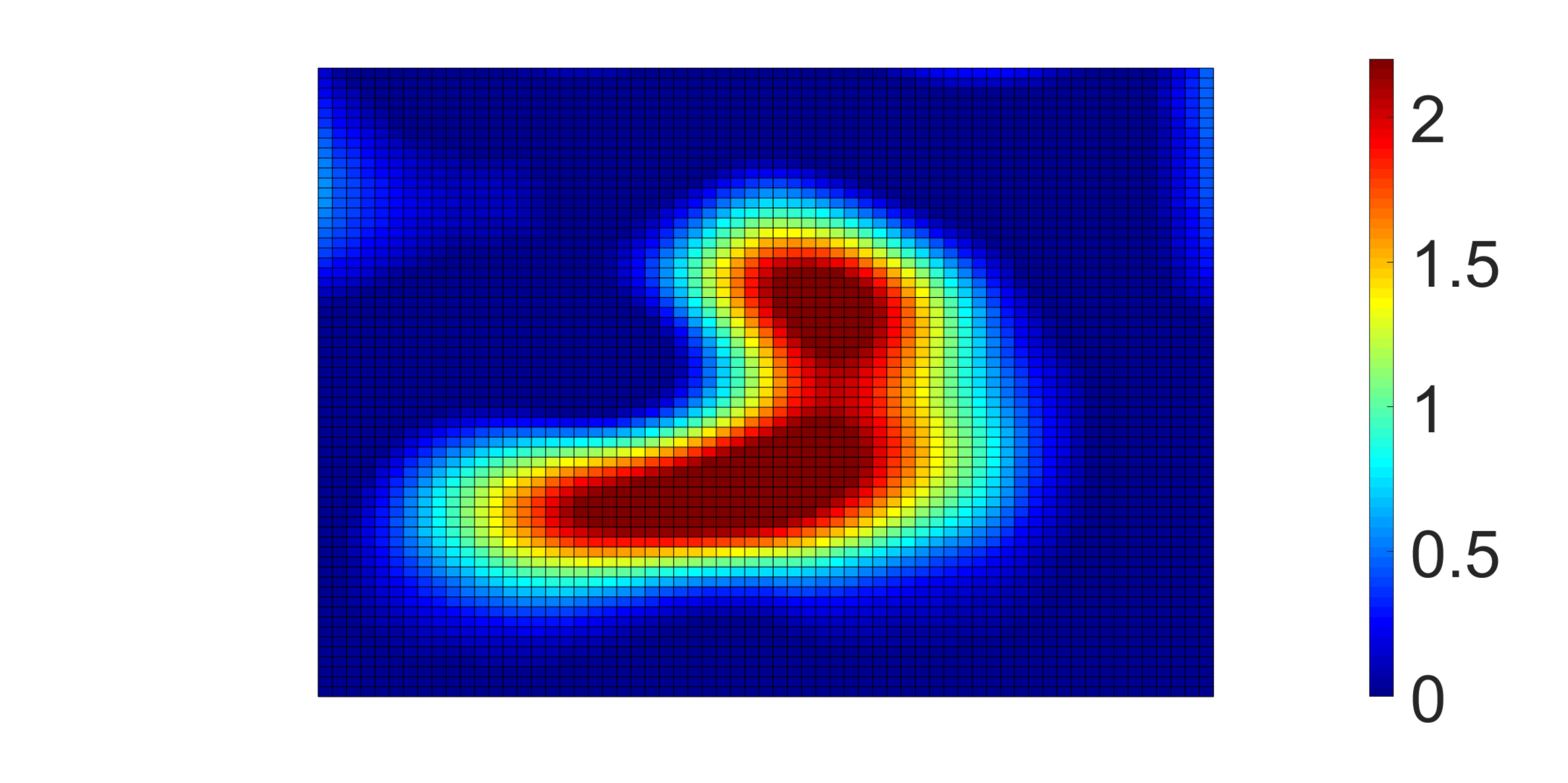}
                    }
                    \hfill
               \end{minipage}\\[1em]
               \begin{minipage}{\textwidth}
                    \hfill
                    \subfigure[Reconstruction of  $\Im  k^2 q$ in the case of $\tilde{\Lambda}$.]{
                        \includegraphics[width=0.47\textwidth]{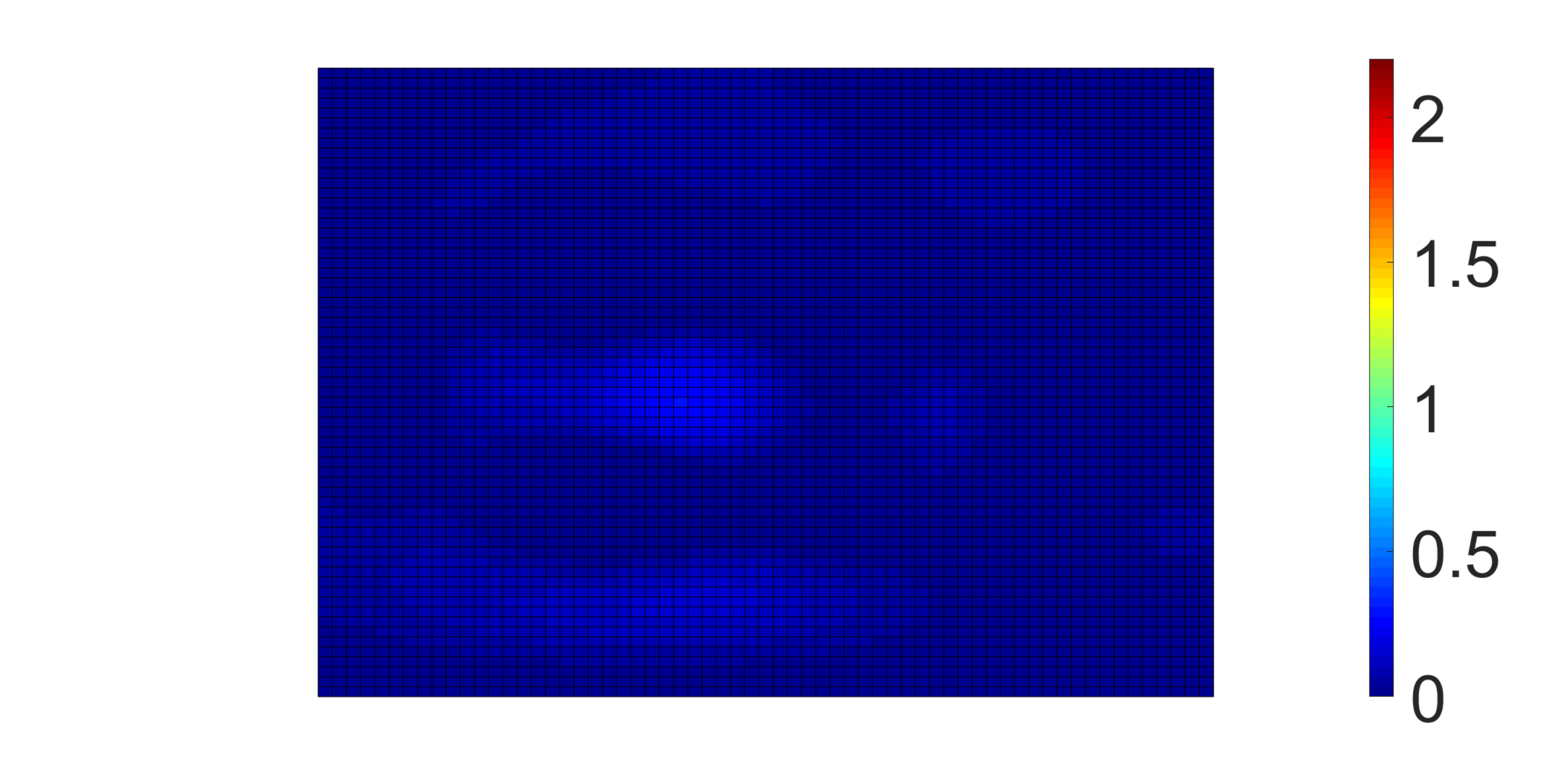}
                    }
                    \hfill
                    \subfigure[Reconstruction of  $\Im  k^2 q$ in the case of $\tilde{\mathcal{S}}$.]{
                        \includegraphics[width=0.47\textwidth]{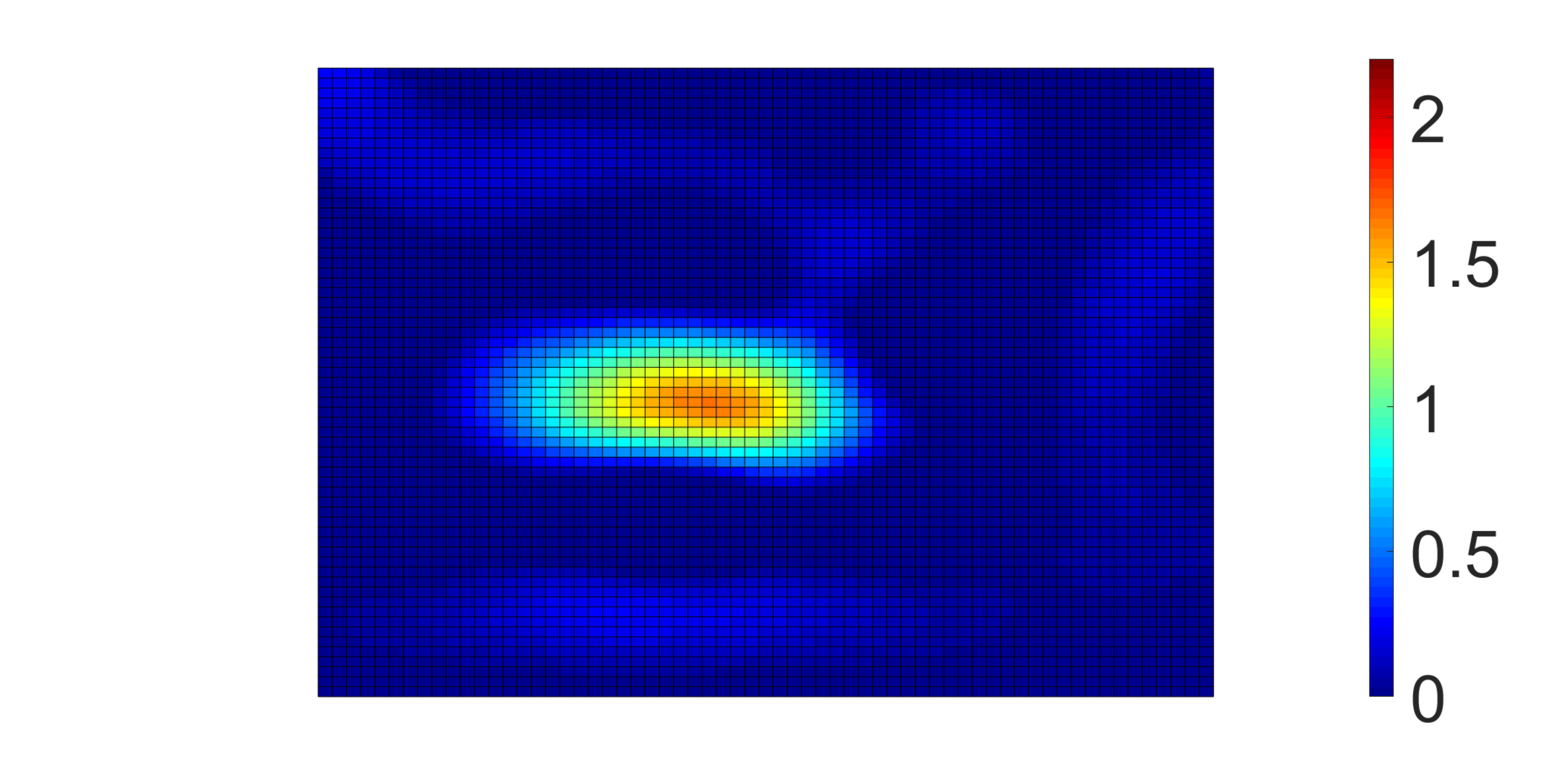}
                    }
                    \hfill
                    \caption{Reconstruction for both measurement operators $\tilde{\Lambda}$ and $\tilde{\mathcal{S}}$ ($d=2$).}
                    \label{image_reconstruction}
                \end{minipage}
            \end{figure}

            For the three dimensional example, we use $4096$ cells for $\Omega^R_0$, $256$ cells of $I_{}$ and a cut-off for the Rayleigh boundary condition of $|j| \leq 30$, $j \in \Z^2$. For the right hand side, we split the domain $\Omega_0^{R_0}$ into $8$ cubes, and for the data, we added $5 \%$ of unified distributed noise. 
            The relative reconstruction error in the case of $\tilde{\Lambda}$ is about $51.3783\%$ (compare \Cref{image_reconstruction3D}).
            
            \begin{figure}[!ht]
                \begin{minipage}{\textwidth}
                    \hspace*{0.1\textwidth}
                    \hfill
                    \subfigure[Exact refractive index $\Re  k^2 n_p^2$.]{
                        \includegraphics[width=0.3\textwidth]{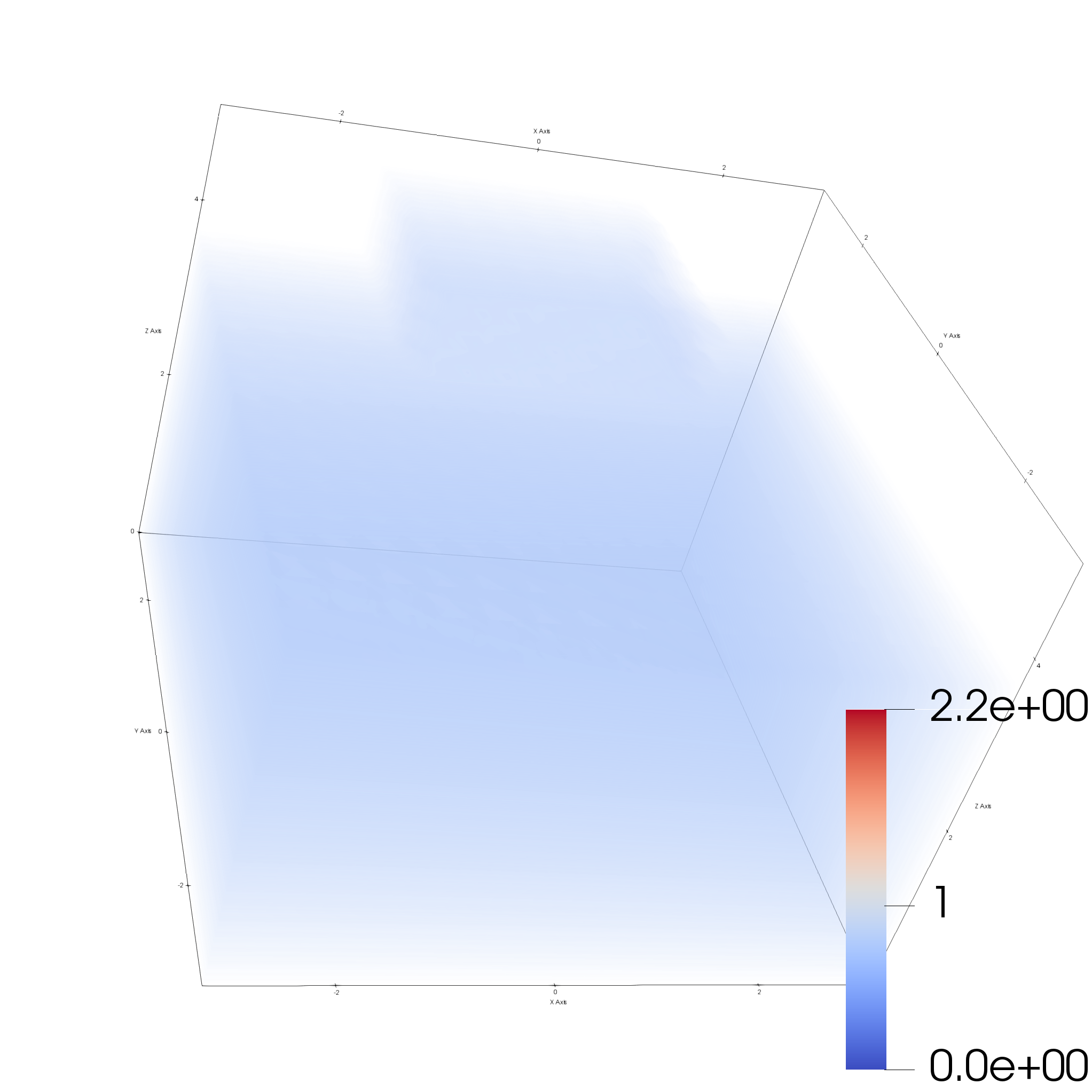}
                    }
                    \hspace*{0.1\textwidth}
                    \subfigure[Exact refractive index $\Im  k^2 n_p^2$.]{
                        \includegraphics[width=0.3\textwidth]{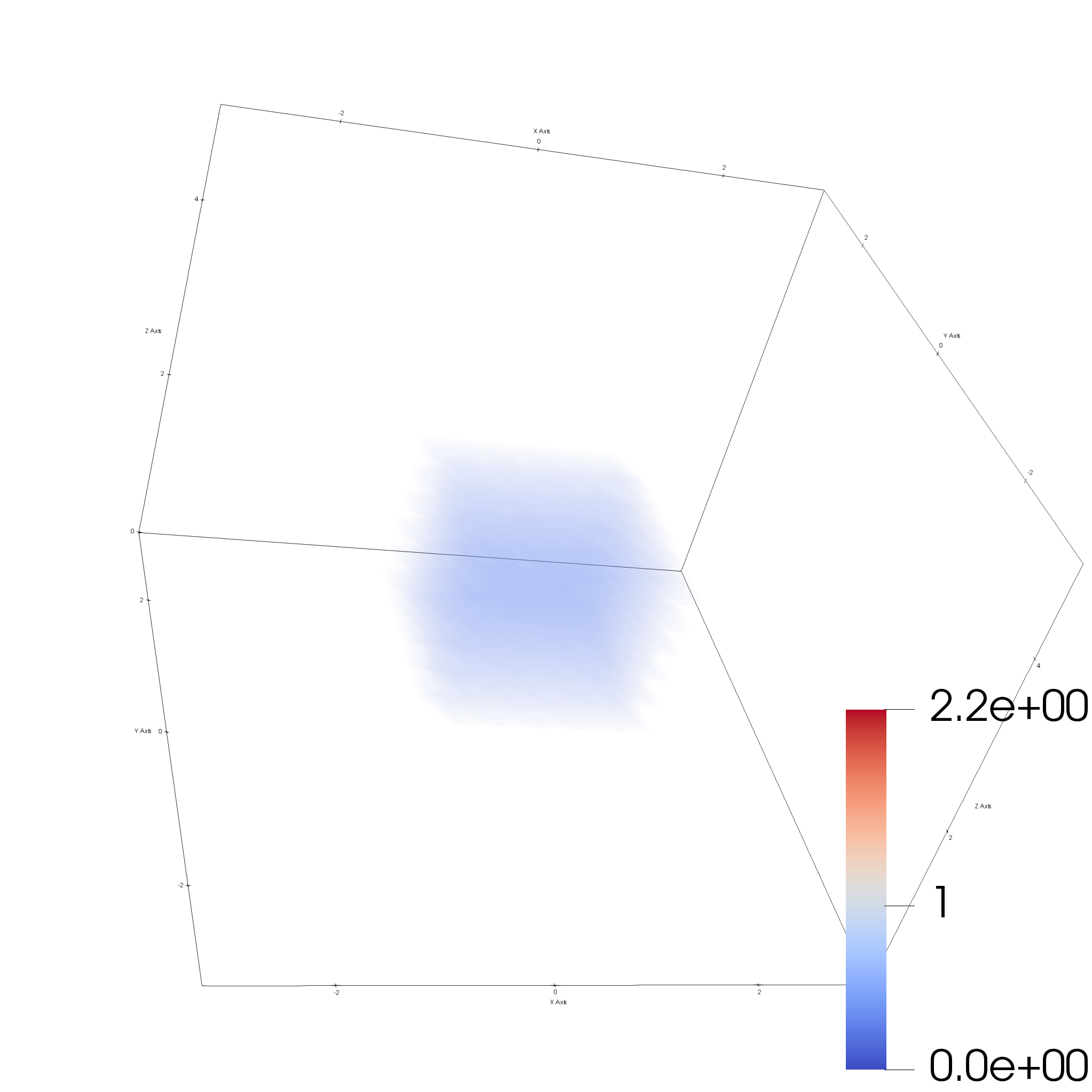}
                    }
                    \hfill
                    \hspace*{0.1\textwidth}
               \end{minipage}\\[1em]
               \begin{minipage}{\textwidth}
                    \hspace*{0.1\textwidth}
                    \hfill
                    \subfigure[Exact perturbation $\Re k^2 q$.]{
                        \includegraphics[width=0.3\textwidth]{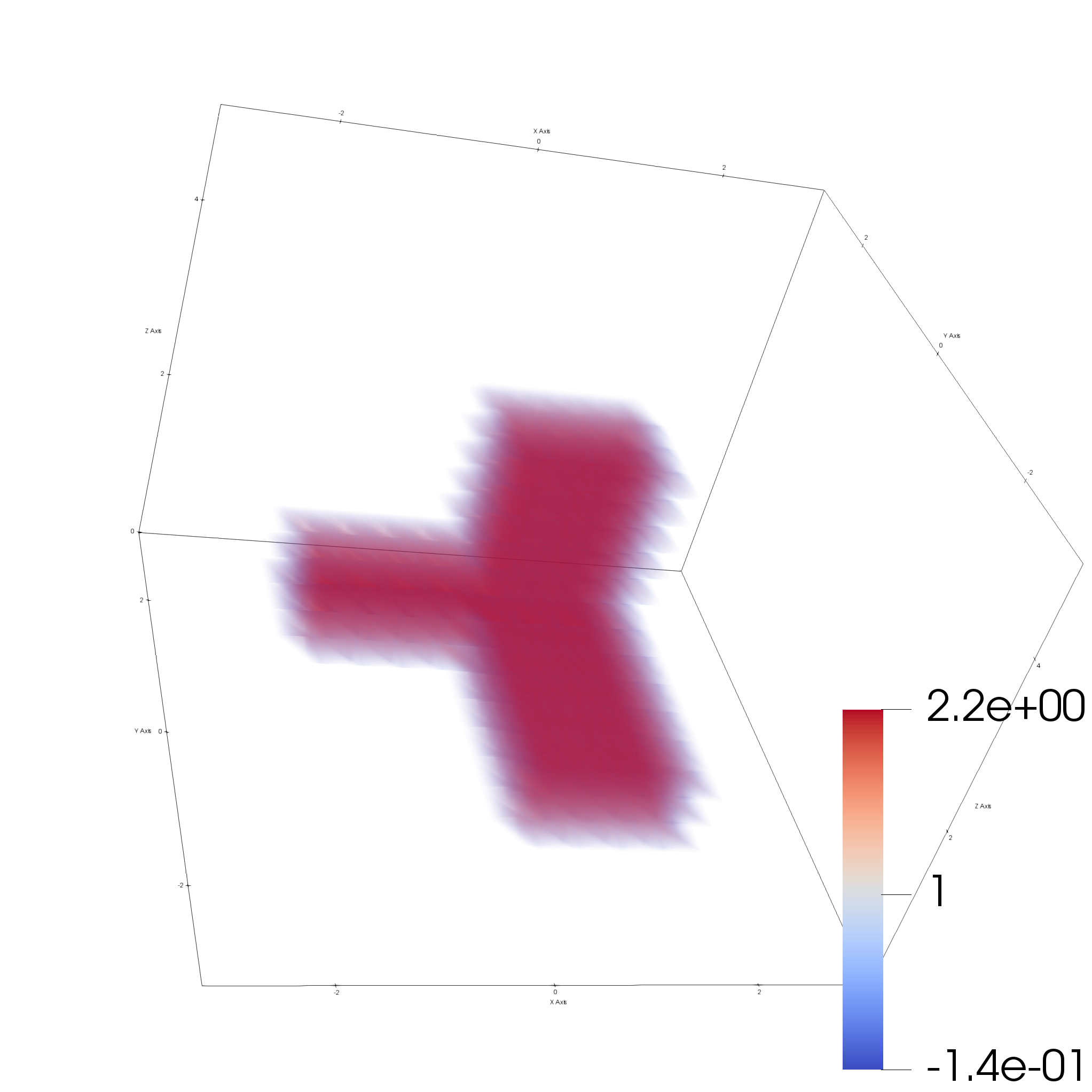}
                    }
                    \hspace*{0.1\textwidth}
                    \subfigure[Reconstruction of $\Re k^2 q$ in the case of $\tilde{\Lambda}$.]{
                        \includegraphics[width=0.3\textwidth]{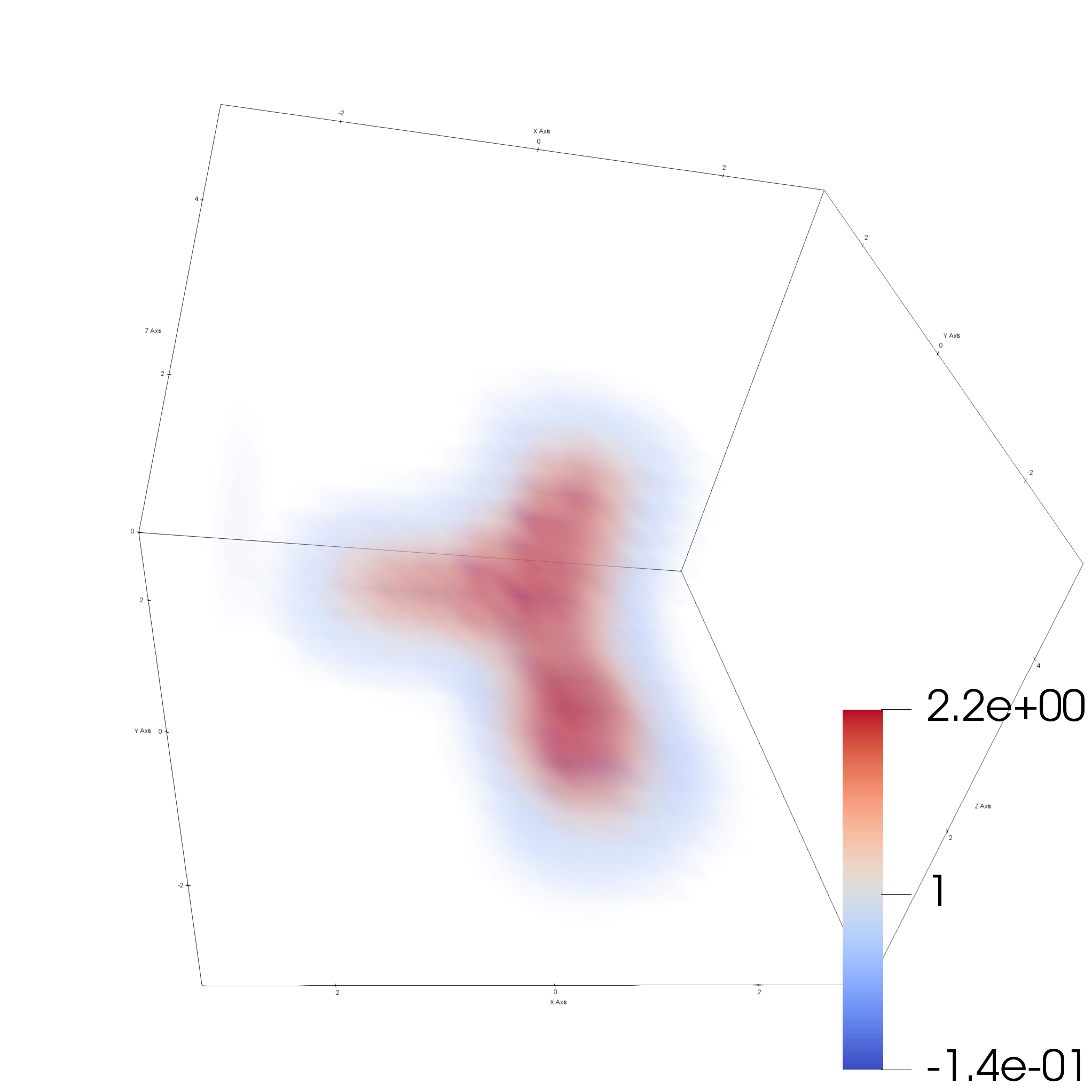}
                    }
                    \hfill
                    \hspace*{0.1\textwidth}
                    \caption{Reconstruction for $d=3$.}
                    \label{image_reconstruction3D}
                \end{minipage}
            \end{figure}
    }
    
    \paragraph{\textbf{Acknowledgement}}
    The first author is very grateful for the devoted and generous support of Armin Lechleiter during his master's and PhD program, who, although no longer with us, continues to inspire by his example and dedication to mathematics and teaching.
    
    This project was funded by the Deutsche Forschungsgemeinschaft (DFG, German Research Foundation) - Projektnummer 281474342/GRK2224/1.

    \printbibliography

\end{document}